\documentclass[smallcondensed,numbook]{svjour3}     

\smartqed  

\usepackage[numbers]{natbib}
\usepackage{latexsym}
\usepackage{amsmath, amssymb, colonequals}
\usepackage[ruled,vlined,linesnumbered]{algorithm2e}
\usepackage{graphics, graphicx, url, color, setspace, lineno, enumerate, verbatim}
\usepackage{tikz}
\usepackage{pdflscape}
\usepackage{longtable}
\usepackage{geometry}
\usepackage[english]{babel}
\usepackage{setspace}
\usepackage{bm}
\usepackage{hyperref}
\usepackage{comment}
\usepackage{float}
\usepackage{caption}
\usepackage{subcaption}
\usepackage{pbox}
\usepackage[titletoc,toc,page]{appendix}
\usepackage{multirow}							
\usepackage[misc]{ifsym}						
\usepackage{color,soul}

\usepackage{xr}

\newcommand{\black}[1]{\textcolor{black}{#1}}

\usepackage{setspace}
\usepackage{authblk}

\def\tr{^{\intercal}}

\def\Z{\mathbb{Z}}
\def\Re{{\mathbb R}}
\def\N{\mathbb{N}}

\def\conv{\mathop{\rm conv}}

\def\proj{\mathop{\rm proj}}

\def\Re{{\mathbb R}}

\def\Sol{\mr{Sol}}
\def\lb{\! \downarrow}
\def\ub{\! \uparrow}
\def\BB{B\&B }
\def\SBB{SB\&B }

\newcommand{\vc}[1]{\bm{#1}}	
\newcommand{\mc}[1]{\mathcal{#1}}	
\newcommand{\mr}[1]{\mathrm{#1}}	
\newcommand{\mt}[1]{\mathtt{#1}}	

\journalname{}

\begin{document}
	
	\title{A graphical framework for global optimization of mixed-integer nonlinear programs
		\thanks{This work was supported in part by the AFOSR YIP Grant FA9550-23-1-0183, and the NSF CAREER Grant CMMI-2338641.}}

	\titlerunning{A graphical global solver for MINLPs}        

	\author{
		Danial Davarnia \and
		Mohammadreza Kiaghadi \and
        Junyuan Qiu}

	\authorrunning{Danial Davarnia \and Mohammadreza Kiaghadi \and Junyuan Qiu} 

	\institute{D. Davarnia \at
		Edwardson School of Industrial Engineering, Purdue University, West Lafayette, IN\\
		\email{ddavarn@purdue.edu} \Letter		   
		\and
		M. Kiaghadi \at
		Department of Industrial and Manufacturing Systems Engineering, Iowa State University, Ames, IA\\
		\email{kiaghadi@iastate.edu}
        \and
        J. Qiu \at
		Edwardson School of Industrial Engineering, Purdue University, West Lafayette, IN\\
		\email{qiu311@purdue.edu}
	}

	\date{Received: date / Accepted: date}
	
	\maketitle


\begin{abstract}
	While mixed-integer linear programming and convex programming solvers have advanced significantly over the past several decades, solution technologies for general mixed-integer nonlinear programs (MINLPs) have yet to reach the same level of maturity. 
	Various problem structures across different application domains remain challenging to model and solve using modern global solvers, primarily due to the lack of efficient parsers and convexification routines for their complex algebraic representations.
	In this paper, we introduce a novel graphical framework for globally solving MINLPs based on decision diagrams (DDs), which enable the modeling of complex problem structures that are intractable for conventional solution techniques.
	We describe the core components of this framework, including a graphical reformulation of MINLP constraints, convexification techniques derived from the constructed graphs, efficient cutting plane methods to generate linear outer approximations, and a spatial branch-and-bound scheme with convergence guarantees. 
	In addition to providing a global solution method for tackling challenging MINLPs, our framework addresses a longstanding gap in the DD literature by developing a general-purpose DD-based approach for solving general \black{bounded} MINLPs.
	To demonstrate its capabilities, we apply our framework to solve instances from one of the most difficult classes of unsolved test problems in the MINLP Library, which are otherwise inadmissible for state-of-the-art global solvers.

\end{abstract}

\section{Introduction} \label{sec:introduction}

Optimization solvers play a crucial role in advancing mathematical optimization by bridging theoretical breakthroughs with computational power to solve real-world problems. While mixed-integer linear programming and convex programming solvers have made significant strides over the past few decades, solution techniques for general MINLPs still face substantial challenges. 
Some of these challenges arise from problem structures that fall outside the framework of current solvers, primarily due to the absence of appropriate parsers for specific algebraic representations found in various application domains.
\black{Examples include error functions in statistical models for portfolio optimization and risk management applications \cite{dahl1989some}, hyperbolic functions in learning models for compressor power in artificial intelligence applications \cite{schweidtmann2019deterministic}, cross-entropy functions in information theory and econometrics applications \cite{golan1996maximum,judge2011information,robinson2001updating}, and gamma functions in quantum mechanics applications \cite{ogura1999post};} see Section~\ref{sec:computation} for a detailed discussion for such applications. 
Even when problems fall within the modeling capabilities of modern solvers, several classes with complex structures suffer from weak approximations and poor solution performance. 
As a result, there remains an ongoing need to develop global solution algorithms that mitigate such limitations of traditional techniques.

\smallskip
In this paper, we introduce a novel graphical framework to globally solve general \black{bounded} MINLPs. 
The basis of this framework is formed by DDs, where the underlying problems are formulated through special-structured graphs. 
These graphs draw out data structures and variable interactions that often remain latent in the classical algebraic representation of constraints.
This intrinsic feature enables DDs to model a broad array of functional forms, including nonconvex, nonsmooth, and even black-box types, that are intractable by standard solution techniques.
Numerous computational studies suggest that DD-based algorithms can improve the solution time and quality compared to the outcome of modern solvers.
Despite the success of DDs in various application areas, they have never been used to globally solve general MINLPs. 
As a result, the framework proposed in this paper marks the first solution technology for MINLPs based on DDs that departs from traditional algebraic approaches to global optimization by capturing the graphical structure of the formulation.

\subsection{Related Work} \label{subsec:literature}

MINLPs are considered one of the most challenging classes of optimization problems, as they involve a combination of continuous and discrete variables along with nonlinear relationships in the constraints and/or objective function \cite{belotti2013mixed,lee2011mixed}. 
As a result, globally solving MINLP formulations can be a daunting task, even for the most advanced optimization solvers.
The prevalent framework to solve these problems globally is the \textit{spatial branch-and-bound} \cite{smith:pa:1999}, which relies on successive bound-reductions and convexification routines \cite{belotti:le:li:ma:wa:2009,castro2017spatial}.
The most common convexification routine is the \textit{factorable decomposition}, where complicated terms are decomposed into simpler components with known convex relaxations \cite{McCormick1976Computability,khajavirad:mi:sa:2014}. 
As the leading commercial global solver, BARON \cite{sahinidis:1996} employs these strategies to handle a wide range of MINLP structures.
At its core, BARON integrates range-reduction methods with branch-and-bound techniques to guide the search towards a globally optimal solution \cite{puranik:sa:2017}.
To accelerate the search process, cutting planes and range-contraction methods have been proposed \cite{ryoo:sa:1996,tawarmalani:sahinidis:2005}. 
Additional boosting techniques for special-structured problems include domain reduction for separable concave programs \cite{shectman:sa:1998}, consistency methods for mixed-binary problems \cite{davarnia:ra:ho:2022}, envelope construction for bilinear terms \cite{gupte2013solving,davarnia:ri:ta:2017,khademnia:da:2024}, and decomposition strategies for multilinear sets \cite{bao:kh:sa:ta:2015,delpia:kha:2018,luedtke2012some,ryoo:sa:2001}.
The concept of convex extension has also been introduced to achieve tighter relaxations for lower semi-continuous functions \cite{tawarmalani:sahinidis:2002} and fractional programs \cite{tawarmalani:sahinidis:2001}.
As a common class of MINLPs, nonconvex quadratically-constrained quadratic programs have attracted considerable research efforts, ranging from polyhedral approximations \cite{saxena2010convex} to semi-definite relaxations \cite{bao:sa:ta:2011}.  
Open-source solvers such as SCIP \cite{achterberg:2009} and COUENNE \cite{belotti:2009} utilize constraint programming and polyhedral approximations, respectively, to globally solve MINLPs.
Despite advancements in these global solvers, there remains a lack of capable solvers for handling MINLPs with complex functional forms, such as hyperbolic trigonometric terms, as noted in \cite{bienstcok:es:ge:2020}.

\smallskip

While the aforementioned approaches target MINLPs with general structures, a significant body of literature focuses on a special class of MINLPs where the underlying functions are convex. 
For these problems, the convexity property allows the design of an \textit{outer approximation} scheme that can converge to an optimal solution without relying on spatial branch-and-bound techniques \cite{duran:gr:1986,hijazi2014outer,muts2020decomposition}.
Outer approximation methods establish a refinement framework that recursively constructs and solves mixed-integer linear approximations of the problem; see \cite{bonami:ki:li:2012} for a survey on such methods. 
Various solvers have been developed based on this framework, 
\black{including BONMIN \cite{bonami:lee:2007}, DICOPT \cite{grossmann2002gams}, and SHOT \cite{lundell2022supporting}}.
A recent review of the computational performance of these solvers can be found in \cite{kronqvist:be:lu:gr:2017}.
However, when applied to nonconvex MINLPs, these solvers often provide local solutions with no guarantee of global optimality. 
Other local solvers commonly used to solve continuous relaxations of MINLPs include IPOPT \cite{IPOPT:2006}, SNOTP \cite{SNOPT:2005}, and KNITRO \cite{KNITRO:2006}.
Most of these approaches use interior-point methods to find a locally-optimal solution; see \cite{tits:wa:ba:ur:la:2003,IPOPT:2006}.
In contrast to these outer approximation and local solvers, our proposed solution method leverages outer approximation to find global optimal solutions for nonconvex MINLPs.

\smallskip
The core structure of our solution methodology in this paper is formed by DDs.
DDs were introduced in \cite{hadzic:ho:2006} as an alternative modeling tool for certain classes of combinatorial problems.
Later, \cite{andersen:ha:ho:ti:2007} proposed the concept of relaxed DDs to mitigate the exponential growth in DD size when modeling large-scale discrete problems.
Since then, significant efforts have been dedicated to enhancing DD performance in discrete optimization problems; see \cite{vanhoeve:2024} for a tutorial on DDs, \cite{bergman:ci:va:ho:2016} for an introduction to DD modeling, and \cite{castro:ci:be:2022} for a recent survey.
Thanks to their promising performance, DDs have been applied across a wide range of application areas, including healthcare \cite{bergman:ci:2018}, supply chain management \cite{bergman:ci:va:ho:2016-1}, and transportation \cite{salemi:da:2023}. 
Other avenues of research in the DD community include cutting plane theory \cite{davarnia:va:2020,tjandraatmadja:va:2019}, multi-objective and Lagrangian optimization \cite{bergman2015lagrangian}, post-optimality analysis \cite{serra2019compact}, sub-optimality and dominance detection \cite{coppe:gi:sc:2024}, integrated search tree \cite{gonzalez2020integrated}, two-stage stochastic programs \cite{lozano:sm:2018,salemi:davarnia:2022}, and sequence alignment \cite{hosseininasab2021exact}.
The novel perspective that DDs offer for modeling optimization problems has propelled DD-based solution methods into the spotlight in recent years. 
In this paper, we extend DD scope further by leveraging their unique structural properties to develop a general-purpose global framework for solving complex MINLPs.


\subsection{Contributions} \label{subsec:contributions}

\textbf{Contributions to the MINLP literature.} 
As discussed in Section~\ref{subsec:literature}, the global MINLP solution methods are based mainly on the algebraic representation of the constraints.
Nonconvex nonlinear terms are typically decomposed into simpler forms and convexified individually, making these techniques highly dependent on the specific algebraic form and properties of the functions involved. 
For example, hyperbolic and trigonometric terms, which are common in applications ranging from artificial intelligence to energy systems, 
\black{remain challenging and are sometimes inadmissible for leading global optimization solvers, such as BARON and SCIP, as they often require specialized convexification machinery and may otherwise lead to weak relaxations or poor computational performance. We also note that there exist global optimization approaches beyond factorable programming, including Lipschitzian methods \mbox{\cite{gablonsky2001locally}} and black-box optimization frameworks \mbox{\cite{audet2006mesh}}, that can handle MINLPs with complex functional forms under different assumptions.}

\smallskip
\black{Complementary to these approaches}, our DD-based framework offers significant flexibility with respect to the functional forms of the problem. 
The graphical nature of DDs enables direct evaluation and efficient relaxation of the underlying terms without decomposition during DD construction. 
As a result, our framework can be effectively applied to a wide range of MINLPs containing highly nonlinear, nonconvex, nonsmooth, and even black-box functions—many of which are intractable or poorly handled by modern global solvers. 
This paper presents a novel global solution method for MINLPs that is rooted in the graphical structure of the problem, rather than its algebraic representation.

\medskip

\noindent
\textbf{Contributions to the DD literature.}
Despite the successful application of DDs to various optimization problems over the past two decades, two significant limitations in their applicability have persisted: (i) DDs have primarily been applied to problems with special structures, and (ii) DDs were originally limited to modeling discrete programs. 
These limitations, recognized in \cite{bergman:ci:va:ho:2016}, have posed a significant barrier to the widespread adoption of DDs, highlighting the need for a general-purpose DD technology to solve MINLPs. 
In \cite{davarnia:va:2020}, the authors addressed challenge (i) by introducing an outer approximation framework that tightens relaxations of integer nonlinear programs using DDs. To tackle challenge (ii), \cite{davarnia:2021} proposed a novel DD-based methodology that obtains strong dual bounds for continuous programs. Subsequent works, including \cite{salemi:davarnia:2022,salemi:da:2023}, extended the scope of DDs through a DD-based Benders decomposition approach, which enabled their application to mixed-integer linear programs in energy systems and transportation.

\smallskip
While these efforts have extended the applicability of DDs to new problem classes, they have primarily focused on optimization problems with separable functions due to the extreme complexity of analyzing cases where the underlying functions involve non-separable terms. 
As a result, these methods have been largely limited in their application to general MINLPs, which often include non-separable constraints.
In this paper, we bridge this gap by introducing the first DD-based framework capable of modeling general nonlinear functional forms with non-separable structures. 
This significantly expands the scope of DD applications to a broad range of previously inaccessible MINLPs.
Furthermore, prior works in the literature were not designed to obtain global optimal solutions for general MINLPs, lacking key elements such as algorithmic architectures and convergence mechanisms required for handling such broader problem structures. 
In this paper, we address this gap by introducing a novel branch-and-cut framework that leverages the structure of DDs at every stage of the solution process, from constructing efficient relaxations, to generating cutting planes and outer approximations, to performing branch-and-bound with convergence guarantees.
As a result, this work establishes the first general-purpose DD-based solution method for globally solving MINLPs.

\medskip
The remainder of this paper is organized as follows.
In \black{Section~\ref{sec:prelim}}, we introduce the structure of the MINLP under study and outline the main steps of our proposed solution framework.
Section~\ref{sec:construction} provides background on DDs and details the algorithms used to construct DDs for different problem structures, representing relaxations of the MINLP. 
Additionally, we present strategies for calculating bounds for DDs and analyze the runtime complexity of the algorithms.
In Section~\ref{sec:OA}, we demonstrate how the constructed DDs can be used to generate linear outer approximations for the MINLP via various cut-generation methods employed within a separation oracle.
Section~\ref{sec:SBB} introduces a spatial branch-and-bound scheme designed to refine these outer approximations, with convergence guarantees to a global optimal solution for the MINLP.
To evaluate the effectiveness of the framework, we present computational experiments on benchmark MINLP instances in \black{Section~\ref{sec:computation}}.
Concluding remarks are provided in Section~\ref{sec:conclusion}.

\medskip

\noindent \textbf{Notation.} 
We denote the vectors by bold letters.
For any $k \in \N$, we define $[k] = \{1, 2, \dotsc, k\}$.
Given a vector $\vc{x} \in \Re^n$, we refer to a sub-vector of \black{$\vc{x}$} that includes variables with indices in $J \subseteq [n]$ as $\vc{x}_J$.
We use calligraphic font to describe sets.
Given a set $\mc{P} \subseteq \{(\vc{x}, \vc{y}) \in \Re^{n+m}\}$, we refer to the convex hull of $\mc{P}$ by $\conv(\mc{P})$.  
We denote by $\proj_{\vc{x}}(\mc{P})$ the projection of $\mc{P}$ onto the space of $\vc{x}$ variables.
For a nested sequence $\{\mc{P}^j\}$ of sets $\mc{P}^j \subseteq \Re^n$ for $j \in \N$, we denote by $\{\mc{P}^j\} \searrow \mc{P}$ the fact that this sequence converges (in the Hausdorff sense) to a set $\mc{P} \subseteq \Re^n$.
Given a closed interval $\mc{D} \subseteq \Re$, we refer to its lower and upper bound as $\mc{D} \lb$ and $\mc{D} \ub$, respectively. 
To distinguish notation, we will represent the elements of a DD using `typewriter' font.
In particular, we define a DD as $\mt{D} = (\mathtt{U}, \mathtt{A}, \mathtt{l}(.))$.
In this definition, the nodes of the DD are represented by $\mt{u} \in \mt{U}$ with state value $\mt{s}(\mt{u})$, and the arcs of DD are denoted by $\mt{a} \in \mt{A}$ with label $\mt{l}(\mt{a})$. 
We refer to the tail and the head nodes of an arc $\mt{a} \in \mt{A}$ as $\mt{t}(\mt{a})$ and $\mt{h}(\mt{a})$, respectively. 

\section{Problem Definition} \label{sec:prelim}

Consider the MINLP 
\begin{subequations} \label{eq:MINLP}
	\begin{align}
		\zeta^* = \max \quad &\vc{c} \tr \vc{x} \label{eq:MINLP-1} \\
		\text{s.t.} \quad &g^k(\vc{x}) \leq b_k, &\forall k \in K \label{eq:MINLP-2}\\
		&x_i \in \mathcal{D}_i, &\forall i \in I \cup C \label{eq:MINLP-3}
	\end{align}
\end{subequations}
where $g^k(\vc{x}): \prod_{i \in I \cup C} \mc{D}_i \to \Re$ for $k \in K$ is a general mixed-integer nonlinear function that is well-defined and bounded over the domain of variables described in \eqref{eq:MINLP-3}.
In the above model, $I$ and $C$ are the index sets for integer and continuous variables, respectively.
Further, $\mathcal{D}_i := [\mathcal{D}_i \lb, \mathcal{D}_i \ub] \cap \Z$ represents the bounded domain for integer variable $x_i$ with $i \in I$, and $\mathcal{D}_i := [\mathcal{D}_i \lb, \mathcal{D}_i \ub]$ represents the bounded domain interval for continuous variable $x_i$ for $i \in C$. 
This definition implies that for $i \in I$, $\mathcal{D}_i \lb, \mathcal{D}_i \ub \in \Z$.
Define the feasible region of constraint $k \in K$ over the variables' domain as 
\begin{equation}
	\mathcal{G}^k = \left\{\vc{x} \in \prod_{i \in I \cup C} \mc{D}_i \, \middle| \, g^k(\vc{x}) \leq b_k \right\}. \label{eq:OA_constraint}
\end{equation}

\smallskip
We outline the main steps of our solution method for globally solving \eqref{eq:MINLP-1}--\eqref{eq:MINLP-3} in Algorithm~\ref{alg:global}. 
Following this, we offer a high-level overview of the most critical components of the algorithm. 
Commonly used elements in global solvers, along with algorithmic settings, are not detailed here as they are thoroughly covered in the relevant literature; see Section~\ref{sec:introduction} for examples.

\smallskip 
This method utilizes a branch-and-bound (B\&B) tree, where each node represents a specific restriction of the feasible region of \eqref{eq:MINLP-1}--\eqref{eq:MINLP-3} induced by partitioning the variable domains.
The first node (root) of the \BB tree is created in line 1, where the original variable domains $\mc{D}_i$ for $i \in [n]$ are considered.
In line 2, the function $\text{Stop\_Flag}$ checks whether a stopping criterion has been met, signaling the termination of the algorithm.
These criteria might include factors such as a time limit, iteration count, remaining optimality gap, number of open nodes in the \BB tree, and more.
If the stopping criteria are not met, the algorithm continues in line 3 by selecting the next open node in the \BB tree for processing. 
This node can be chosen using any well-known strategy, such as depth-first, breadth-first, or best-bound.
Each node in the \BB tree is associated with a linear programming (LP) relaxation, $LP$, of the problem, defined by the variable domains corresponding to that node, along with any linear inequalities inherited from its parent node, if applicable.
This LP model is solved in line 4 to obtain an optimal solution $\vc{x}^*$.
If the problem is infeasible or unbounded, the primal bound $\underline{\zeta}$ or the dual bound $\overline{\zeta}$ will be updated accordingly in line 9.
Subsequently, for each constraint $k \in K$, the loop in lines 5--8 is executed.
First, it is checked whether the current solution $\vc{x}^*$ satisfies the constraint \eqref{eq:MINLP-2} for the current $k$.
If it does not, the oracle $\mt{Construct\_DD}$ is called to build a DD that represents the solutions (or a relaxation of the solutions) to $\mathcal{G}^k$.
Next, a linear outer approximation for the solutions of the constructed DD associated with $\mathcal{G}^k$ is generated by calling the oracle $\mt{Outer\_Approx}$ to separate point $\vc{x}^*$ from $\conv(\mathcal{G}^k)$.
The constraints obtained from this outer approximation are then added to $LP$.
Once these steps are completed for all constraints $k \in K$, the augmented LP model $LP$ is resolved.
Depending on the optimal value obtained from this model, the bounds $\underline{\zeta}$ and $\overline{\zeta}$ are updated accordingly.
In line 10, the function $\text{Prune\_Node}$ checks whether the current node can be pruned.
The pruning rules applied in this function may include standard rules such as pruning by feasibility, pruning by infeasibility, and pruning by bound, as well as more advanced rules like pruning due to inconsistency \cite{morrison:she:ja:2016}.   
If the node is not pruned, the oracle $\mt{Branch}$ is called to perform the branching operation.
This operation includes identifying a variable to branch on and determining the branching value, which may utilize any well-known techniques such as most fractional rule, strong branching, and pseudocost branching; see \cite{bonami2013branching,morrison:she:ja:2016}.
Following the branching operation, two new child nodes are created and added to the \BB tree. 
The bounds of the selected variable are updated for each node based on the branching value. 
Once this step is completed, the recursive steps of the algorithm are repeated until the stopping criteria in line 2 are met.
At this point, the algorithm terminates and returns the calculated primal and dual bounds, which can be used to calculate the remaining optimality gap achieved by the algorithm.

\smallskip
As noted in the description of Algorithm~\ref{alg:global}, three key oracles, namely $\mt{Construct\_DD}$, $\mt{Outer\_Approx}$, and $\mt{Branch}$, constitute the backbone of our solution framework.
The following three sections are dedicated to explaining these oracles in detail.

\begin{algorithm}[!ht]
	\caption{A high-level structure of the graphical global optimization framework}
	\label{alg:global}
	\KwData{MINLP of the form \eqref{eq:MINLP-1}--\eqref{eq:MINLP-3}}
	\KwResult{A lower bound $\underline{\zeta}$ and upper bound $\overline{\zeta}$ for $\zeta^*$}
	
	create the root node of the \BB tree that includes the original domain of variables
	
	\While{$\text{Stop\_Flag} = \text{False}$}
	{
		{select an open node in the \BB tree}\\
		
		{solve an initial LP relaxation $LP$ of the problem at this node to obtain an optimal solution $\vc{x}^*$}		
		
		\ForAll{$k \in K$}
		{
			
			\If{\black{$g^k(\vc{x}^*) > b_k$}}
			{
				{call $\mt{Construct\_DD}$ to construct a DD respreseting $\mathcal{G}^k$}\\
				{call $\mt{Outer\_Approx}$ to create a linear outer approximation of $\conv(\mathcal{G}^k)$ based on its associated DD and add it to $LP$}				
			}		
		}
		
		{solve $LP$ to update $\underline{\zeta}$ and $\overline{\zeta}$, if possible}\\
		
		\If{$\text{Prune\_Node} = \text{True}$}
		{
			{prune the current node in the \BB tree}							
		}
		\Else
		{
			{call $\mt{Branch}$ to perform branching and create children nodes to be added to the \BB tree}
		}
	}

\end{algorithm}

\section{DD Construction} \label{sec:construction}

In this section, we discuss the oracle $\mt{Construct\_DD}$ in Algorithm~\ref{alg:global} by outlining the steps involved in constructing DDs that represent relaxations of the set $\mc{G}^k$ for $k \in K$.
To simplify notation, we will omit the index $k$ whenever the results apply to any constraint, regardless of its specific index.
We begin with a brief background on using DDs for optimization in Section~\ref{sub: Background on DDs}. 
In Section~\ref{sub:nonseparable}, we present methods for constructing DDs for a general form of $\mc{G}^k$ that may involve non-separable functions.
Section~\ref{sub:merging} includes algorithms for merging nodes at layers of a DD to obtain relaxed DDs of a desired size.
In Section~\ref{sub:lower bound}, we introduce a strategy for calculating state values for the nodes at the DD layers to ensure that the resulting DDs provide a valid relaxation for the underlying set.
Taking all these components into account, we present the time complexity results for the proposed DD construction algorithms in Section~\ref{sub:complexity}. 
In the remainder of this paper, we assume, without loss of generality, a variable ordering in $I \cup C$ corresponding to the layers of the DD, denoted by $[n] = \{1, \dotsc, n\}$.

\subsection{Background on DDs} \label{sub: Background on DDs}

In this section, we present basic definitions and results relevant to our DD analysis.
A DD $\mt{D}$ is a directed acyclic graph denoted by the triple $(\mt{U},\mt{A}, \mt{l}(.))$ where $\mt{U}$ is a node set, $\mt{A}$ is an arc set, and $\mt{l}: \mt{A} \to \Re$ is an arc label mapping for the graph components.
This DD is composed of $n \in \N$ arc layers $\mt{A}_1, \mt{A}_2, \dots, \mt{A}_n$, and $n+1$ node layers $\mt{U}_1,\mt{U}_2,\dots,\mt{U}_{n+1}$.
The node layers $\mt{U}_1$ and $\mt{U}_{n+1}$ contain the root $\mt{r}$ and the terminal $\mt{t}$, respectively. 
In any arc layer $j \in [n]$, an arc $\mt{a} \in \mathcal \mt{A}_j$ is directed from the tail node $\mt{t}(\mt{a}) \in \mt{U}_j$ to the head node $\mt{h}(\mt{a}) \in \mt{U}_{j+1}$. 
The \textit{width} of $\mt{D}$ is defined as the maximum number of nodes at any node layer $\mt{U}_j$. 
DDs have been traditionally used to model a bounded integer set $\mathcal P \subseteq \Z^n$ such that each $\mt{r}$-$\mt{t}$ arc-sequence (path) of the form $(\mt{a}_1, \dotsc, \mt{a}_n) \in \mt{A}_1 \times \dotsc \times \mt{A}_n$ encodes a point $\vc x \in \mathcal P$ where $\mt{l}(\mt{a}_j) = x_j$ for $j \in [n]$, that is $\vc x$ is an $n$-dimensional point in $\mathcal P$ whose $j$-th coordinate is equal to the label value $\mt{l}(\mt{a}_j)$ of the arc $\mt{a}_j$.
For such a DD, we have $\mathcal P = \Sol(\mt{D})$, where $\Sol(\mt{D})$ represents the set of all $\mt{r}$-$\mt{t}$ paths.

\smallskip
As outlined above, DDs have traditionally been employed to model and solve discrete optimization problems. 
For instance, they have been extensively used to address combinatorial problems with special structures, such as stable set, set covering, and matching; see \cite{bergman:ci:va:ho:2016}.
Recently, through a series of works \cite{davarnia:2021,salemi:davarnia:2022,salemi:da:2023,davarnia:ki:2025}, the application of DD-based optimization has been extended to mixed-integer programs. 
This extension has enabled applications in new domains, ranging from energy systems to transportation, involving a combination of discrete and continuous variables.
However, these works primarily focus on problems with separable functions.
In this paper, we address this limitation by considering general functional forms in MINLP constraints that may involve non-separable terms, requiring a significantly more complex analysis.
Our results in this paper unify and significantly expand upon the methods developed in the prior works, leading to a general-purpose global solution framework for MINLPs, which integrates all essential components, from convexification to spatial branch-and-bound techniques with convergence guarantees.

\smallskip
The following result presents a technique known as \textit{arc reduction}, which reduces the size of a decision diagram (DD) while preserving the convex hull of its solution set.
Originally introduced in \cite{davarnia:2021} for modeling continuous sets using DDs, 
\black{a similar reduction technique can be applied here. In particular, since the technique depends only on the labels of parallel arcs between connected nodes, it can be used in mixed-integer sets without relying on whether variables are continuous or discrete.}
Consider the following definitions for a DD $\mt{D} = (\mt{U}, \mt{A}, \mt{l}(.))$.
For each pair $(\mt{u}, \mt{v})$ of connected nodes of $\mt{D}$ with $\mt{u} \in \mt{U}_i$ and $\mt{v} \in \mt{U}_{i+1}$ for some $i \in [n]$, define $l^{\max}(\mt{u}, \mt{v})$ to be the maximum label of all arcs connecting $\mt{u}$ and $\mt{v}$, i.e., $l^{\max}(\mt{u}, \mt{v}) = \max\{ \mt{l}(\mt{a}) \, | \, \mt{a} \in \mt{A},  \mt{t}(\mt{a}) = \mt{u}, \mt{h}(\mt{a}) = \mt{v} \}$.
Similarly, define $l^{\min}(\mt{u}, \mt{v}) = \min\{ \mt{l}(\mt{a}) \, | \, \mt{a} \in \mt{A},  \mt{t}(\mt{a}) = \mt{u}, \mt{h}(\mt{a}) = \mt{v} \}$ to be the minimum label of all arcs connecting $\mt{u}$ and $\mt{v}$.

\begin{proposition} \label{prop:remove arc}
	Consider a DD $\mt{D} = (\mt{U}, \mt{A}, \mt{l}(.))$.
	Let $\bar{\mt{D}} = (\bar{\mt{U}}, \bar{\mt{A}}, \bar{\mt{l}}(.))$ be a DD obtained from $\mt{D}$ by removing every arc $\mt{a} \in \mt{A}$ such that $l^{\min}(\mt{t}(\mt{a}), \mt{h}(\mt{a})) < \mt{l}(\mt{a}) < l^{\max}(\mt{t}(\mt{a}), \mt{h}(\mt{a}))$. 
	Then, $\conv(\Sol(\bar{\mt{D}})) = \conv(\Sol(\mt{D}))$.
\end{proposition}

\begin{proof}
	We prove the result by showing $\conv(\Sol(\bar{\mt{D}})) \subseteq \conv(\Sol(\mt{D}))$ and $\conv(\Sol(\bar{\mt{D}})) \supseteq \conv(\Sol(\mt{D}))$.
	The forward inclusion is straightforward as the arcs in $\bar{\mt{D}}$ are a subset of the arcs in $\mt{D}$ by definition, which implies that the root-terminal paths in $\bar{\mt{D}}$ are a subset of the root-terminal paths in $\mt{D}$.
	Therefore, $\Sol(\bar{\mt{D}}) \subseteq \Sol(\mt{D})$, which yields $\conv(\Sol(\bar{\mt{D}})) \subseteq \conv(\Sol(\mt{D}))$.
	
	\smallskip
	For the reverse inclusion, consider a point $\vc{x} \in \conv(\Sol(\mt{D}))$.
	It follows that there exists a collection of $\mt{r}$-$\mt{t}$ paths of the form $\mt{P}^j = (\mt{a}^j_1, \dotsc, \mt{a}^j_n)$ of $\mt{D}$ for $j \in [p]$ for some $p \in \N$, each encoding a point $\vc{x}^j = (\mt{l}(\mt{a}^j_1), \dotsc, \mt{l}(\mt{a}^j_n))$, such that $\vc{x} = \sum_{j=1}^p \lambda_j \vc{x}^j$ with $\sum_{j=1}^p \lambda_j = 1$ and $\lambda_j \geq 0$ for all $j \in [p]$.
	Next, construct the sets $S_j$ for each $j \in [p]$ that consists of points $\hat{\vc{x}}^{j, k} \in \Re^n$ for $k \in [q^j]$ for some $q^j \in \N$ such that $\hat{x}^{j,k}_i \in \big\{l^{\min}(\mt{t}(\mt{a}^j_i), \mt{h}(\mt{a}^j_i)), l^{\max}(\mt{t}(\mt{a}^j_i), \mt{h}(\mt{a}^j_i)) \big\}$ for each $i \in [n]$.
	There are a maximum of $2^n$ such points. 
	Each such point $\hat{\vc{x}}^{j,k}$ corresponds to an $\bar{\mt{r}}$-$\bar{\mt{t}}$ path of $\bar{\mt{D}}$ because all of its arcs are maintained for this DD by construction, implying that $\hat{\vc{x}}^{j,k} \in \Sol(\bar{\mt{D}})$.
	As a result, we can write $\vc{x}^j = \sum_{k =1}^{q_j} \mu^{j,k} \hat{\vc{x}}^{j, k}$ for some $\vc{\mu}$ such that $\sum_{k =1}^{q_j} \mu^{j,k} = 1$ and $\mu^{j,k} \geq 0$ for all $k \in [q_j]$.
	Using this relation for all $j \in [p]$, we can write $\vc{x} = \sum_{j=1}^p \lambda_j \vc{x}^j = \sum_{j=1}^p \lambda_j ( \sum_{k =1}^{q_j} \mu^{j,k} \hat{\vc{x}}^{j, k})$ where $\hat{\vc{x}}^{j,k} \in \Sol(\bar{\mt{D}})$ for each $j \in [p]$ and $k \in [q_j]$ with $\lambda_j \mu^{j,k} \geq 0$ and $\sum_{j=1}^p \sum_{k =1}^{q_j} \lambda_j \mu^{j,k} = 1$.  
	Therefore, $\vc{x} \in \conv(\Sol(\bar{\mt{D}}))$.
		\qed       
\end{proof}

Proposition~\ref{prop:remove arc} suggests that when multiple parallel arcs exist between two nodes in a DD, we can retain only the arcs with the minimum and maximum label values, removing all others. 
This operation preserves the convex hull of the DD's solution set. 
For the remainder of this paper, we will assume that this technique is applied wherever applicable.

\subsection{Relaxed DD for Non-Separable Constraints}  \label{sub:nonseparable}

In this section, we outline the steps for constructing a DD that represents a relaxation of the constraint set $\mc{G} = \left\{\vc{x} \in \prod_{i=1}^{n} \mc{D}_i \, \middle| \, g(\vc{x}) \leq b \right\}$ where the underlying function $g(\vc{x})$ may contain non-separable terms.
We refer to this general function form as ``non-separable'' since it encompasses models with separable functions as a special case.
As noted in Section~\ref{sec:introduction}, the DD results presented in the literature \cite{davarnia:va:2020,davarnia:2021} have mainly been developed for cases where the functions involved in the constraints are separable, i.e., they can be written as $g(\vc{x}) = \sum_{i=1}^n g_i(x_i)$.
In these cases, following the DD construction procedure in \cite{davarnia:va:2020,davarnia:2021}, the state values calculated for a node layer of the DD depend only on the state values of the nodes in the previous layer and the labels of the arcs in that layer, thereby satisfying a Markovian property. 
This property streamlines the construction of the DD.
However, in the non-separable case, the state values calculated for a node layer can be impacted by the arc labels from any of the preceding layers, due to the interactions between variables in the non-separable terms. As a result, constructing a relaxed DD in this context is significantly more challenging.
In what follows, we illustrate the methods for constructing such relaxed DDs.
These results also unify the methods developed in \cite{davarnia:va:2020} for discrete problems and \cite{davarnia:2021} for continuous problems, extending them to represent the mixed-integer case for a general MINLP. 


\begin{definition} \label{def:backtrack}
	Consider a DD $\mt{D} = (\mt{U}, \mt{A}, \mt{l}(.))$ with the variable ordering $1, \dotsc, n$.
	Given a layer $i \in [n]$ and a node $\mt{v} \in \mt{U}_i$, we define $\mt{A}_j(\mt{v})$ to be the set of arcs in arc layer $j \in [i-1]$ that lie on a path from the root node to node $\mt{v}$.
	Further, we define $\mc{D}_j(\mt{v}) = [\mc{D}_j(\mt{v}) \lb, \mc{D}_j(\mt{v}) \ub]$ for $j \in [i-1]$ to be \textit{the sub-domain of variable $x_j$ relative to node $\mt{v}$}.
	In this definition, the domain lower bound is calculated as $\mc{D}_j(\mt{v}) \lb = \min_{\mt{a} \in \mt{A}_j(\mt{v})} \{\mt{l}(\mt{a})\}$ and the domain upper bound is calculated as $\mc{D}_j(\mt{v}) \ub = \max_{\mt{a} \in \mt{A}_j(\mt{v})} \{\mt{l}(\mt{a})\}$. 
\end{definition}

Determining the variable sub-domains relative to a node requires accounting for all paths in the DD that pass through that node. 
This can become computationally intensive, especially for large-scale DDs. 
The following proposition offers an efficient method for calculating these bounds by leveraging the top-down construction process of DDs.


\begin{proposition} \label{prop:relative domain}
	Consider a DD $\mt{D} = (\mt{U}, \mt{A}, \mt{l}(.))$ with the variable ordering $1, \dotsc, n$.
	Given a node layer $i \in [n]$ and a node $\mt{v} \in \mt{U}_i$, we can calculate $\mc{D}_j(\mt{v}) = [\mc{D}_j(\mt{v}) \lb, \mc{D}_j(\mt{v}) \ub]$ for $j \in [i-1]$ as follows:
	\begin{itemize}
		\item[(i)] If $j = i-1$:
		\begin{subequations}	
			\begin{align}
				& \mc{D}_j(\mt{v}) \lb = \min_{\mt{a} \in \delta^-(\mt{v})} \big\{ \mt{l}(\mt{a}) \big\} \label{eq:relative domain:1}\\
				& \mc{D}_j(\mt{v}) \ub = \max_{\mt{a} \in \delta^-(\mt{v})} \big\{ \mt{l}(\mt{a}) \big\}. \label{eq:relative domain:2}
			\end{align}
		\end{subequations}	
		\item[(ii)] If $j < i-1$:
		\begin{subequations}	
			\begin{align}
				& \mc{D}_j(\mt{v}) \lb = \min_{\mt{a} \in \delta^-(\mt{v})} \big\{ \mc{D}_j(\mt{t}(\mt{a})) \lb \big\} \label{eq:relative domain:3}\\
				& \mc{D}_j(\mt{v}) \ub = \max_{\mt{a} \in \delta^-(\mt{v})} \big\{ \mc{D}_j(\mt{t}(\mt{a})) \ub \big\}. \label{eq:relative domain:4}
			\end{align}
		\end{subequations}
	\end{itemize}
\end{proposition}


\begin{proof}
	We prove the results for the lower bound equations \eqref{eq:relative domain:1} and \eqref{eq:relative domain:3} as the proof arguments for the upper bound equations are similar.
	\begin{itemize}		
		\item[(i)] Assume that $j = i-1$. 
		In this case, node $\mt{v} \in \mt{U}_{i}$ must be the head node of each arc $\mt{a} \in \mt{A}_j(\mt{v})$ that is on a path from the root node of $\mt{D}$ to node $\mt{v}$.
		That is, $\mt{A}_j(\mt{v}) \subseteq \delta^-(\mt{v})$.
		On the other hand, using the DD structure that each node at layer $i \in \{2, \dotsc, n-1\}$ is connected to at least one node in the previous layer, we obtain that each arc $\mt{a} \in \delta^-(\mt{v})$ must be on a path from the root node to $\mt{v}$, \textit{i.e.}, $\mt{A}_j(\mt{v}) \supseteq \delta^-(\mt{v})$.
		Therefore, $\mt{A}_j(\mt{v}) = \delta^-(\mt{v})$. 
		As a result, by definition of $\mc{D}_j(\mt{v}) \lb$, we have $\mc{D}_j(\mt{v}) \lb = \min_{\mt{a} \in \mt{A}_j(\mt{v})} \{\mt{l}(\mt{a})\} = \min_{\mt{a} \in \delta^-(\mt{v})} \big\{ \mt{l}(\mt{a}) \big\}$.		
		
		\item[(ii)] Assume that $j < i-1$.
		For the forward direction, consider arc $\mt{a} \in \mt{A}_j(\mt{v})$, \textit{i.e.}, $\mt{a}$ is on a path from the root node of $\mt{D}$ to node $\mt{v}$.
		This path passes through some node $\mt{u}$ in the node layer $i-1$ that is connected to node $\mt{v}$ via arc $\dot{\mt{a}}$, \textit{i.e.}, $\mt{u} = \mt{t}(\dot{\mt{a}})$ and $\dot{\mt{a}} \in \delta^-(\mt{v})$.
		As a result, $\mt{a} \in \mt{A}_j(\mt{u})$.
		Using this relation, we obtain that $\mt{A}_j(\mt{v}) \subseteq \bigcup_{\dot{\mt{a}} \in \delta^-(\mt{v})} \mt{A}_j(\mt{t}(\dot{\mt{a}}))$.
		Therefore, $\mc{D}_j(\mt{v}) \lb = \min_{\mt{a} \in \mt{A}_j(\mt{v})} \{\mt{l}(\mt{a})\} \geq \min_{\dot{\mt{a}} \in \delta^-(\mt{v})} \big\{ \min_{\hat{\mt{a}} \in \mt{A}_j(\mt{t}(\dot{\mt{a}}))} \{\mt{l}(\hat{\mt{a}})\} \big\} = \min_{\dot{\mt{a}} \in \delta^-(\mt{v})} \big\{ \mc{D}_j(\mt{t}(\dot{\mt{a}})) \lb \big\}$, where the first equality follows from the definition of $\mc{D}_j(\mt{v}) \lb$, the inequality follows from the previously derived inclusion argument, and the last equality follows from the definition of $\mc{D}_j(\mt{t}(\dot{\mt{a}})) \lb$.
		By design, node $\mt{v}$ is connected to a node $\mt{u}$ in the previous node layer $\mt{U}_{i-1}$.
		For the reverse direction, assume that $\min_{\mt{a} \in \delta^-(\mt{v})} \big\{ \mc{D}_j(\mt{t}(\mt{a})) \lb \big\}$ is achieved by arc $\mt{a}^* \in \delta^-(\mt{v})$, \textit{i.e,.} $\mc{D}_j(\mt{t}(\mt{a}^*)) \lb = \mt{l}(\bar{\mt{a}})$ for some $\bar{\mt{a}} \in \mt{A}_j(\mt{t}(\mt{a}^*))$.
		This means that $\bar{\mt{a}}$ is on a path from the root node of $\mt{D}$ to node $\mt{t}(\mt{a}^*)$, which is the tail node of arc $\mt{a}^*$ that connects this node to $\mt{v}$.
		As a result, $\bar{\mt{a}}$ is on a path from the root node of $\mt{D}$ to node $\mt{v}$, \textit{i.e.}, $\bar{\mt{a}} \in \mt{A}_j(\mt{v})$.
		We can write that $\mc{D}_j(\mt{v}) \lb = \min_{\mt{a} \in \mt{A}_j(\mt{v})} \{\mt{l}(\mt{a})\} \leq \mt{l}(\bar{\mt{a}}) = \mc{D}_j(\mt{t}(\mt{a}^*)) \lb = \min_{\mt{a} \in \delta^-(\mt{v})} \big\{ \mc{D}_j(\mt{t}(\mt{a})) \lb \big\}$, where the first equality follows from the definition of $\mc{D}_j(\mt{v}) \lb$, the inequality holds because $\bar{\mt{a}} \in \mt{A}_j(\mt{v})$, and the second and third equalities follow from the assumptions.
		Combining both directions, we obtain the \black{equation~\eqref{eq:relative domain:3}}.		
	\end{itemize}
		\qed       
\end{proof}

The recursive relations \eqref{eq:relative domain:1}--\eqref{eq:relative domain:2} and \eqref{eq:relative domain:3}--\eqref{eq:relative domain:4} provide an efficient method to compute the values of $\mc{D}_j(\mt{v})$ during the construction of DD layers by updating the lower and upper bound values of its interval as each new node is created.
The next corollary gives the complexity of these calculations.
The proof is omitted, as it directly follows from the application of the recursive formulas across consecutive DD layers.

\begin{corollary} \label{cor:relative domain complexity}
	Consider a DD $\mt{D} = (\mt{U}, \mt{A}, \mt{l}(.))$ with the variable ordering $1, \dotsc, n$.
	Given a node layer $i \in [n]$ and an arc layer $j \in [i-1]$, the sub-domain of variable $x_j$ relative to all nodes in $\mt{U}_i$ can be calculated in $\mc{O}(\sum_{k=j}^{i-1} |\mt{A}_k|)$ using the recursive relations \eqref{eq:relative domain:1}--\eqref{eq:relative domain:2} and \eqref{eq:relative domain:3}--\eqref{eq:relative domain:4}.
		\qed
\end{corollary}	

Corollary~\ref{cor:relative domain complexity} shows that the sub-domains of a particular variable in layer $j$ relative to all nodes in layer $i$, for any $j < i$, can be computed in linear time in the number of arcs between layers $j$ and $i$.
This method makes the calculation of these interval values simple and fast during the DD construction process.

\smallskip
Next, we introduce an algorithm (Algorithm~\ref{alg:nonseparable DD}) to build relaxed DDs for non-separable functions.
Consider $\mc{G} = \left\{\vc{x} \in \prod_{i=1}^{n} \mc{D}_i \, \middle| \, g(\vc{x}) \leq b \right\}$, where $g(\vc{x}) = \sum_{j=1}^q g_j(\vc{x}_{H_j})$ such that $g_j(\vc{x}_{H_j}):\Re^{|H_j|} \to \Re$ is a non-separable function that contains variables with indices in $H_j \subseteq [n]$.	
For each $j \in [q]$, define $H_j^{\text{max}} = \max_{k \in H_j}\{k\}$, which represents the last DD layer that involves a variable in the non-separable function $g_j(\vc{x}_{H_j})$.
It is clear from this definition that, when $g(\vc{x})$ is separable, for each $j \in [q]$, we have $H_j = \{l_j\}$ for some $l_j \subseteq [n]$, and $H_j^{\text{max}} = l_j$.

\smallskip
In the description of Algorithm~\ref{alg:nonseparable DD}, we define $L_i$, for each $i \in [n]$, to be the index set for \textit{sub-domain partitions} $\mc{D}^j_i := [\mc{D}^j_i \lb, \mc{D}^j_i \ub]$ for $j \in L_i$, which collectively span the entire domain of variable $x_i$, i.e., $\bigcup_{j \in L_i} \mc{D}^j_i = \mc{D}_i$.
If $x_i$ is integer, a sub-domain partition $\mc{D}^j_i$ may consist of integral numbers within an interval, i.e., $\mc{D}^j_i := [\mc{D}^j_i \lb, \mc{D}^j_i \ub] \cap \Z$.

\begin{algorithm}[!ht]
	\caption{Relaxed DD for a non-separable constraint}
	\label{alg:nonseparable DD}
	\KwData{Set $\mc{G} = \left\{\vc{x} \in \prod_{i=1}^{n} \mc{D}_i \, \middle| \, g(\vc{x}) \leq b \right\}$, where $g(\vc{x}) = \sum_{k=1}^q g_k(\vc{x}_{H_k})$ where $g_k(\vc{x}_{H_k}):\Re^{|H_k|} \to \Re$ is a non-separable function that contains variables with indices in $H_k$, and the sub-domain partitions $\mc{D}^j_i$ with $j \in L_i$ for $i \in [n]$}
	\KwResult{A DD $\mt{D}$}
	
	{create the root node $\mt{r}$ in the node layer $\mt{U}_1$ with state value $\mt{s}(\mt{r}) = 0$}\\ 
	{create the terminal node $\mt{t}$ in the node layer $\mt{U}_{n+1}$}
	
	\ForAll{$i \in [n]$, $\mt{u} \in \mt{U}_i$, $j \in L_i$}
	{
		\ForAll{$k \in [q]$}
		{
			\If{$i = H_k^{\text{max}}$}
			{
				{calculate $\eta_k$ such that $\eta_k \leq g_k(\vc{x}_{H_k})$ for all $x_i \in \mc{D}^j_i$ and $x_l \in \mc{D}_l(\mt{u})$ for $l \in H_k \setminus \{i\}$}				
			}
			\Else
			{
				\black{set $\eta_k = 0$}
			}
		}
		
		{calculate $\xi = \mt{s}(\mt{u}) + \sum_{k=1}^q \eta_k$}
		
		\If{$i < n$}
		{
			{create a node $\mt{v}$ with state value $\mt{s}(\mt{v}) = \xi$ (if it does not already exist) in the node layer $\mt{U}_{i+1}$}\\
			{add two arcs from $\mt{u}$ to $\mt{v}$ with label values $\mc{D}^j_i \lb$ and $\mc{D}^j_i \ub$}\\	
			{update $\mc{D}_k(\mt{v})$ for all $k$ such that $k \in H_l$ for some $l > i$}
		}
		\ElseIf{$\xi \leq b$}
		{
			{add two arcs from $\mt{u}$ to the terminal node $\mt{t}$ with label values $\mc{D}^j_i \lb$ and $\mc{D}^j_i \ub$}				
			
		}
	}
	
\end{algorithm}

\begin{proposition} \label{prop:nonseparable convex hull}
	Consider $\mc{G} = \left\{\vc{x} \in \prod_{i=1}^{n} \mc{D}_i \, \middle| \, g(\vc{x}) \leq b \right\}$, where $g(\vc{x}) = \sum_{j=1}^q g_j(\vc{x}_{H_j})$ such that $g_j(\vc{x}_{H_j}):\Re^{|H_j|} \to \Re$ is a non-separable function that contains variables with indices in $H_j \subseteq [n]$.
	Let $\mt{D}$ be the DD constructed via Algorithm~\ref{alg:nonseparable DD} for some sub-domain partitions $\mc{D}^j_i$ with $j \in L_i$ for $i \in [n]$. Then, $\conv(\mathcal{G}) \subseteq \conv(\Sol(\mt{D}))$.
\end{proposition}

\begin{proof}
    \black{It is sufficient to prove that $\mathcal{G}\subseteq \conv(\Sol(\mt{D}))$. Fix any $\bar{\vc{x}}\in\mathcal{G}$.}
	It follows from the definition of $\mathcal{G}$ that $\sum_{k=1}^q g_k(\bar{\vc{x}}_{H_k}) \leq b$. 	
	For each $i \in [n]$, let $j^*_i$ be the index of a sub-domain partition $\mc{D}^{j^*_i}_i$ in $L_i$ such that $\bar{x}_i \in \mc{D}^{j^*_i}_i$. 
	This index exists because $\bar{x}_i \in \mc{D}_i = \bigcup_{j \in L_i} \mc{D}^j_i$, where the inclusion follows from the fact that $\bar{\vc{x}} \in \mathcal{G}$, and the equality follows from the definition of sub-domain partitions.	 
	
	\smallskip	
	Nest, we show that $\mt{D}$ includes a node sequence $\{\mt{u}_1, \mt{u}_2, \dotsc, \mt{u}_{n+1}\}$, where $\mt{u}_i \in \mt{U}_i$ for $i \in [n+1]$, such that each node $\mt{u}_i$ is connected to $\mt{u}_{i+1}$ via two arcs with labels $\mc{D}^{j^*_i}_i \lb$ and $\mc{D}^{j^*_i}_i \ub$ for each $i \in [n]$.
    \black{The root-to-terminal paths of this node sequence encode the vertices of the hyper-rectangle $\prod_{i=1}^n [\mc{D}_i^{j_i^*}\lb,\mc{D}_i^{j_i^*}\ub]$,
	which contains $\bar{\vc{x}}$ by construction. It then follows that
	$\bar{\vc{x}} \in \conv(\Sol(\mt{D}))$.}

	\smallskip
    \black{The argument rests on the following bound on the state value along this sequence.}
	We use induction on layer $i \in [n]$ to prove $\mt{s}(\mt{u}_i) \leq \sum_{k \in [q]: H_k^{\text{max}} < i}g_k(\bar{\vc{x}}_{H_k})$. 
	In words, the state value of each node $\mt{u}_i$ in the previously picked sequence is no greater than the summation of non-separable function terms $g_k(\bar{\vc{x}}_{H_k})$ whose variables have already been visited in the previous layers of the DD.
	The induction base for $i=1$ follows from line 1 of Algorithm~\ref{alg:nonseparable DD} because $\mt{s}(\mt{u}_1) = \mt{s}(\mt{r}) = 0$, and the fact that $H_k^{\text{max}} > 0$ for all $k \in [q]$, implying that $\sum_{k \in [q]: H_k^{\text{max}} < i}g_k(\bar{\vc{x}}_{H_k}) = 0$ by default.
	For the inductive hypothesis, assume that $\mt{s}(\mt{u}_{\hat{i}}) \leq \sum_{k \in [q]: H_k^{\text{max}} < \hat{i}}g_k(\bar{\vc{x}}_{H_k})$ for $\hat{i} \in [n-1]$.
	For the inductive step, we show that $\mt{s}(\mt{u}_{\hat{i}+1}) \leq \sum_{k \in [q]: H_k^{\text{max}} < \hat{i}+1}g_k(\bar{\vc{x}}_{H_k})$.
	It follows from lines 9--11 of Algorithm~\ref{alg:nonseparable DD} that $\mt{s}(\mt{u}_{\hat{i}+1}) = \mt{s}(\mt{u}_{\hat{i}}) + \sum_{k=1}^q \eta_k$, where $\eta_k$ is a lower bound for $g_k(\vc{x}_{H_k})$ for all $x_{\hat{i}} \in \mc{D}^{j^*_{\hat{i}}}_{\hat{i}}$ and $x_l \in \mc{D}_l(\mt{u}_{\hat{i}})$ for all $l \in H_k \setminus \{\hat{i}\}$.
	Furthermore, we have that $\sum_{k \in [q]: H_k^{\text{max}} < \hat{i}+1}g_k(\bar{\vc{x}}_{H_k}) = \sum_{k \in [q]: H_k^{\text{max}} < \hat{i}}g_k(\bar{\vc{x}}_{H_k}) + \sum_{k \in [q]: H_k^{\text{max}} = \hat{i}}g_k(\bar{\vc{x}}_{H_k})$.
	Therefore, to prove the relation in the inductive step, \black{it remains to verify that} $\sum_{k=1}^q \eta_k \leq \sum_{k \in [q]: H_k^{\text{max}} = \hat{i}}g_k(\bar{\vc{x}}_{H_k})$, because $\mt{s}(\mt{u}_{\hat{i}}) \leq \sum_{k \in [q]: H_k^{\text{max}} < \hat{i}}g_k(\bar{\vc{x}}_{H_k})$ by the inductive hypothesis.
	We can rewrite the right-hand-side $\sum_{k \in [q]: H_k^{\text{max}} = \hat{i}}g_k(\bar{\vc{x}}_{H_k})$ of the above inequality as $\sum_{k=1}^q \phi_k $ where $\phi_k = g_k(\bar{\vc{x}}_{H_k})$ if $H_k^{\text{max}} = \hat{i}$, and $\phi_k = 0$ otherwise.
	As a result, \black{the desired inequality follows once} $\eta_k \leq \phi_k$ for each $k \in [q]$.
	There are two cases.
	
	\smallskip
	For the first case, assume that $H_k^{\text{max}} \neq \hat{i}$. 
	Then, it follows from line 7--8 of Algorithm~\ref{alg:nonseparable DD} that $\eta_k = 0$.
	Additionally, the definition of $\phi_k$ implies that $\phi_k = 0$, which proves the desired inequality.
	
	\smallskip
	For the second case, assume that $H_k^{\text{max}} = \hat{i}$.
	It follows from the arguments in the first paragraph of the proof that $\bar{x}_i \in \mc{D}^{j^*_i}_i$ for all $i \in [n]$.
	Further, we have established that $\{\mt{u}_1, \mt{u}_2, \dotsc, \mt{u}_{\hat{i}}\}$ is a node sequence such that each node $\mt{u}_l$ is connected to $\mt{u}_{l+1}$ via two arcs $\mt{a}^1_l$ and $\mt{a}^2_l$ with labels $\mc{D}^{j^*_l}_l \lb$ and $\mc{D}^{j^*_l}_l \ub$ for $l = 1, \dotsc, \hat{i} - 1$.
	By definition of $H_k^{\text{max}}$, we must have that $H_k \setminus \{\hat{i}\} \subseteq \{1, \dotsc, \hat{i} - 1\}$.
	Therefore, for each $l \in H_k \setminus \{\hat{i}\}$, the arcs $\mt{a}^1_l$ and $\mt{a}^2_l$ must be in $\mt{A}_l(\mt{u}_{\hat{i}})$ because they are on some paths from $\mt{r}$ to $\mt{u}_{\hat{i}}$.
	We can write that $\mc{D}_l(\mt{u}_{\hat{i}}) \lb \, \leq \mc{D}^{j^*_l}_l \lb \, \leq \bar{x}_l$, where the first inequality follows from the definition of $\mc{D}_l(\mt{u}_{\hat{i}})$ in Definition~\ref{def:backtrack}, and the second inequality follows from the fact that $\bar{x}_l \in \mc{D}^{j^*_l}_l$.
	Similarly, we can write that $\mc{D}_l(\mt{u}_{\hat{i}}) \ub \, \geq \mc{D}^{j^*_l}_l \ub \, \geq \bar{x}_l$.
	This yields $\bar{x}_l \in \mc{D}_l(\mt{u}_{\hat{i}})$.
	Additionally, we have argued that $\bar{x}_{\hat{i}} \in \mc{D}^{j^*_{\hat{i}}}_{\hat{i}}$.
	Therefore, it follows from the definition of $\eta_k$ that $\eta_k \leq g_k(\bar{\vc{x}}_{H_k}) = \phi_k$, obtaining the desired inequality.
	
	\smallskip
	\black{It remains to verify that} the condition in line 14 of Algorithm~\ref{alg:nonseparable DD} is satisfied, i.e., $\xi^* = \mt{s}(\mt{u}_{n}) + \sum_{k=1}^q \eta_k^* \leq b$, where we use $^*$ to distinguish the values of $\xi$ and $\eta_k$ calculated at the last layer $i = n$ from those calculated in the previous part for layers $i < n$.
	It follows from the above induction result that $\mt{s}(\mt{u}_{n}) \leq \sum_{k \in [q]: H_k^{\text{max}} < n}g_k(\bar{\vc{x}}_{H_k})$.
	\black{The same argument as in the induction step yields that}, for each $k \in [q]$, we have $\eta^*_k \leq g_k(\bar{\vc{x}}_{H_k})$ if $H_k^{\text{max}} = n$, and $\eta^*_k = 0$ otherwise.
	Aggregating the above two inequalities, we obtain that $\mt{s}(\mt{u}_{n}) + \sum_{k=1}^q \eta^*_k \leq \sum_{k \in [q]: H_k^{\text{max}} < n}g_k(\bar{\vc{x}}_{H_k}) + \sum_{k \in [q]: H_k^{\text{max}} = n}g_k(\bar{\vc{x}}_{H_k})$.
	The right-hand-side of this inequality can be rewritten as $\sum_{k \in [q]}g_k(\bar{\vc{x}}_{H_k})$ because $H_k^{\text{max}} \leq n$ for all $k \in [q]$ by definition.
	On the other hand, the definition of $\mathcal{G}$ implies that $\sum_{k=1}^q g_k(\bar{\vc{x}}_{H_k}) \leq b$.
	Combining the above inequalities, we conclude that $\mt{s}(\mt{u}_{n}) + \sum_{k=1}^q \eta^*_k = \xi^* \leq b$, which satisfies the condition in line 14 of Algorithm~\ref{alg:nonseparable DD}.
	Therefore, line 15 of Algorithm~\ref{alg:nonseparable DD} implies that two arcs with label values $\mc{D}^{j^*_n}_n \lb$ and $\mc{D}^{j^*_n}_n \ub$ connect node $\mt{u}_n$ to the terminal node $\mt{t}$ which can be considered as $\mt{u}_{n+1}$, completing the desired node sequence.
	
	\smallskip	
	Now consider the collection of points $\tilde{\vc{x}}^{\kappa}$ for $\kappa \in [2^n]$ encoded by all paths composed of the above-mentioned pair of arcs with labels $\mc{D}^{j^*_i}_i \lb$ and $\mc{D}^{j^*_i}_i \ub$ between each two consecutive nodes $\mt{u}_i$ and $\mt{u}_{i+1}$ in the sequence $\{\mt{u}_1, \mt{u}_2, \dotsc, \mt{u}_{n+1}\}$.
	Therefore, $\tilde{\vc{x}}^{\kappa} \in \Sol(\mt{D})$ for $\kappa \in [2^n]$.
	It is clear that these points form the vertices of an $n$-dimensional hyper-rectangle defined by $\prod_{i=1}^n [\mc{D}^{j^*_i}_i \lb, \mc{D}^{j^*_i}_i \ub]$.
	By construction, we have that $\bar{\vc{x}} \in \prod_{i=1}^n [\mc{D}^{j^*_i}_i \lb, \mc{D}^{j^*_i}_i \ub]$, i.e., $\bar{\vc{x}}$ is a point inside the above hyper-rectangle.
	As a result, $\bar{\vc{x}}$ can be represented as a convex combination of the vertices $\tilde{\vc{x}}^{\kappa}$ for $\kappa \in [2^n]$ of the hyper-rectangle, yielding $\bar{\vc{x}} \in \conv(\Sol(\mt{D}))$. \\		  
		\qed       
\end{proof}

The following example illustrates the steps of Algorithm~\ref{alg:nonseparable DD} for an MINLP set with a non-separable term.

{
\color{black}
\begin{example} \label{ex:nonseparable}
	Consider a nonconvex MINLP set $\mc{G} = \left\{\vc{x} \in \prod_{i=1}^{3} \mc{D}_i \, \middle| \, -x_1^2 - x_2 - x_1x_3 \leq -2 \right\}$, where $\mc{D}_1 = \{0, 1, 2\}$, $\mc{D}_2 = \{0, 1\}$, and $\mc{D}_3 = [1, 2]$.
	Following the definition of sets studied in this section, we have $g(\vc{x}) = g_1(x_1) + g_2(x_2) + g_3(x_1, x_3)$ where $g_1(x_1) = -x_1^2$, $g_2(x_2) = -x_2$, and $g_3(x_1, x_3)  = -x_1x_3$.
	Further, we obtain $H_1^{\text{max}} = 1$, $H_2^{\text{max}} = 2$, and $H_3^{\text{max}} = 3$.
	The feasible region of $\mc{G}$ can be represented as the following union of hyper-rectangles: $\mc{G} =  \big\{\{1\} \times \{0\} \times [1, 2], \{1\} \times \{1\} \times [1, 2], \{2\} \times \{0\} \times [1, 2], \{2\} \times \{1\} \times [1, 2] \big\}$.
	Therefore, we can obtain the convex hull of this set as $\conv(\mc{G}) = [1, 2] \times [0, 1] \times [1, 2]$.
	
	\smallskip 
	To obtain a relaxation for this set based on DDs, we apply Algorithm~\ref{alg:nonseparable DD} with sub-domain partitions $\mc{D}^1_1 = [0, 0]$, $\mc{D}^2_1 = [1, 1]$, $\mc{D}^3_1 = [2, 2]$ for $x_1$ as a discrete variable, $\mc{D}^1_2 = [0, 0]$, $\mc{D}^2_2 = [1, 1]$ for $x_2$ as a binary variable, and $\mc{D}^1_3 = [1, 2]$ for $x_3$ as a continuous variable.
	Following the steps of the algorithm, we obtain the DD presented in Figure~\ref{fig:nonseparable}, where the numbers next to each arc represent the arc label, and the numbers inside each node show the state value of the node.
	To illustrate the state value computation of non-separable terms, consider the node with state value $0$ at node layer 3, which we refer to as $\mt{u} \in \mt{U}_3$ with $\mt{s}(\mt{u}) = 0$.
	In line 6 of the algorithm, a lower bound $\eta_3$ for $g_3(x_1, x_3)$ must be calculated over the domain $x_3 \in \mc{D}^1_3$ and $x_1 \in \mc{D}_1(\mt{u})$.
	Since there is one path from the root node to $\mt{u}$ with arc label $0$ for variable $x_1$, we obtain $\mc{D}_1(\mt{u}) = [0,0]$.
	Therefore, $\eta_3$ can be set to 0.
	Similarly, we set $\eta_1 = \eta_2 = 0$ in line 8 of the algorithm because $H_1^{\text{max}} \neq 3$ and $H_2^{\text{max}} \neq 3$.
	Therefore, we obtain $\xi = \mt{s}(\mt{u}) + \sum_{i=1}^3 \eta_3 = 0$.
	Since $\xi \nleq -2$, the condition in line 14 of the algorithm is not satisfied, implying that node $\mt{u}$ is not connected to the terminal node.
	
	\smallskip
	Now consider the node with state value $-1$ at node layer 3, which we refer to as $\mt{v} \in \mt{U}_3$ with $\mt{s}(\mt{v}) = -1$.
	In line 6 of the algorithm, a lower bound $\eta_3$ for $g_3(x_1, x_3)$ must be calculated over the domain $x_3 \in \mc{D}^1_3$ and $x_1 \in \mc{D}_1(\mt{v})$.
	Since there are two paths from the root node to $\mt{v}$ with arc labels 0 and 1 for variable $x_1$, we obtain $\mc{D}_1(\mt{v}) = [0,1]$.
	Therefore, $\eta_3$ can be set to $-2$.
	Similarly, we set $\eta_1 = \eta_2 = 0$ in line 8 of the algorithm because $H_1^{\text{max}} \neq 3$ and $H_2^{\text{max}} \neq 3$.
	Therefore, we obtain $\xi = \mt{s}(\mt{v}) + \sum_{i=1}^3 \eta_3 = -3$.
	Since $\xi \leq -2$, the condition in line 14 of the algorithm is satisfied, implying that node $\mt{v}$ is connected to the terminal node by two arcs with labels $0$ and $1$.
	The calculation for other nodes can be carried out similarly, which yields the DD $\mt{D}$ in Figure~\ref{fig:nonseparable_1}.
	Since there is a node at layer 3 that is not connected to the terminal node, we can reduce the size of the DD by removing that node and its incoming arc as it does not lead to a feasible path, which yields a reduced DD $\mt{\bar{D}}$ in Figure~\ref{fig:nonseparable_2}.	
	It is easy to verify that $\conv(\Sol(\mt{\bar{D}})) = \big\{(x_1, x_2, x_3) \in [0, 2] \times [0, 1] \times [1, 2] \, \big| \, x_1 + x_2 \ge 1 \big\}$.
	This set provides a relaxation for the convex hull of the original set $\mc{G}$ with strict inclusion, i.e., $\conv(\mc{G}) \subset \conv(\Sol(\mt{\bar{D}}))$.
	$\blacksquare$
\end{example}
}

\begin{figure}[h]
	\centering
	\begin{subfigure}[b]{0.47\linewidth} 
		\centering
		\includegraphics[scale=1,width=\textwidth]{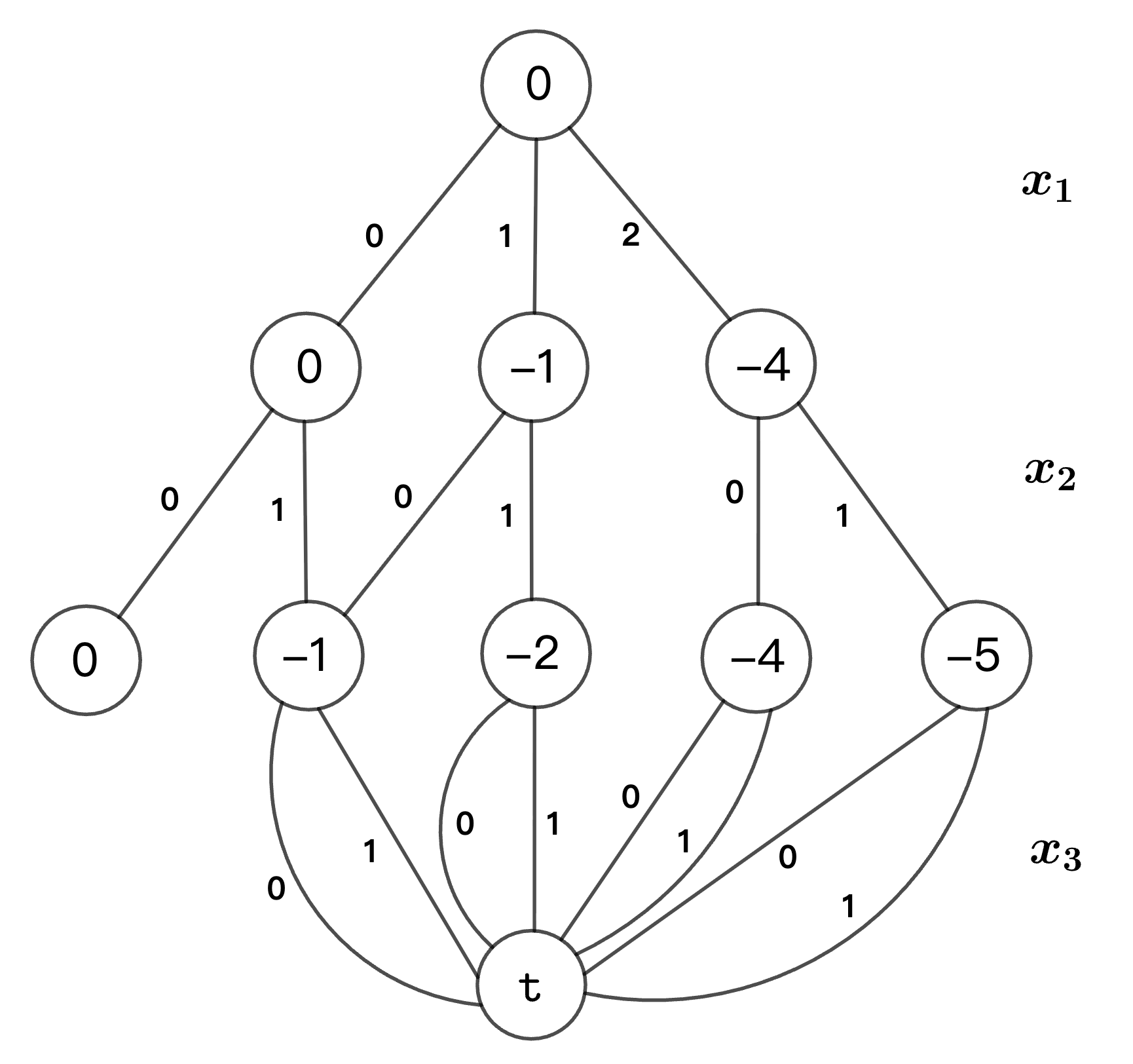}
		\caption{Original DD representation}
		\label{fig:nonseparable_1}
	\end{subfigure}
	\hspace{0.03\linewidth}
	\begin{subfigure}[b]{0.47\linewidth} 
		\centering
		\includegraphics[scale=1,width=\textwidth]{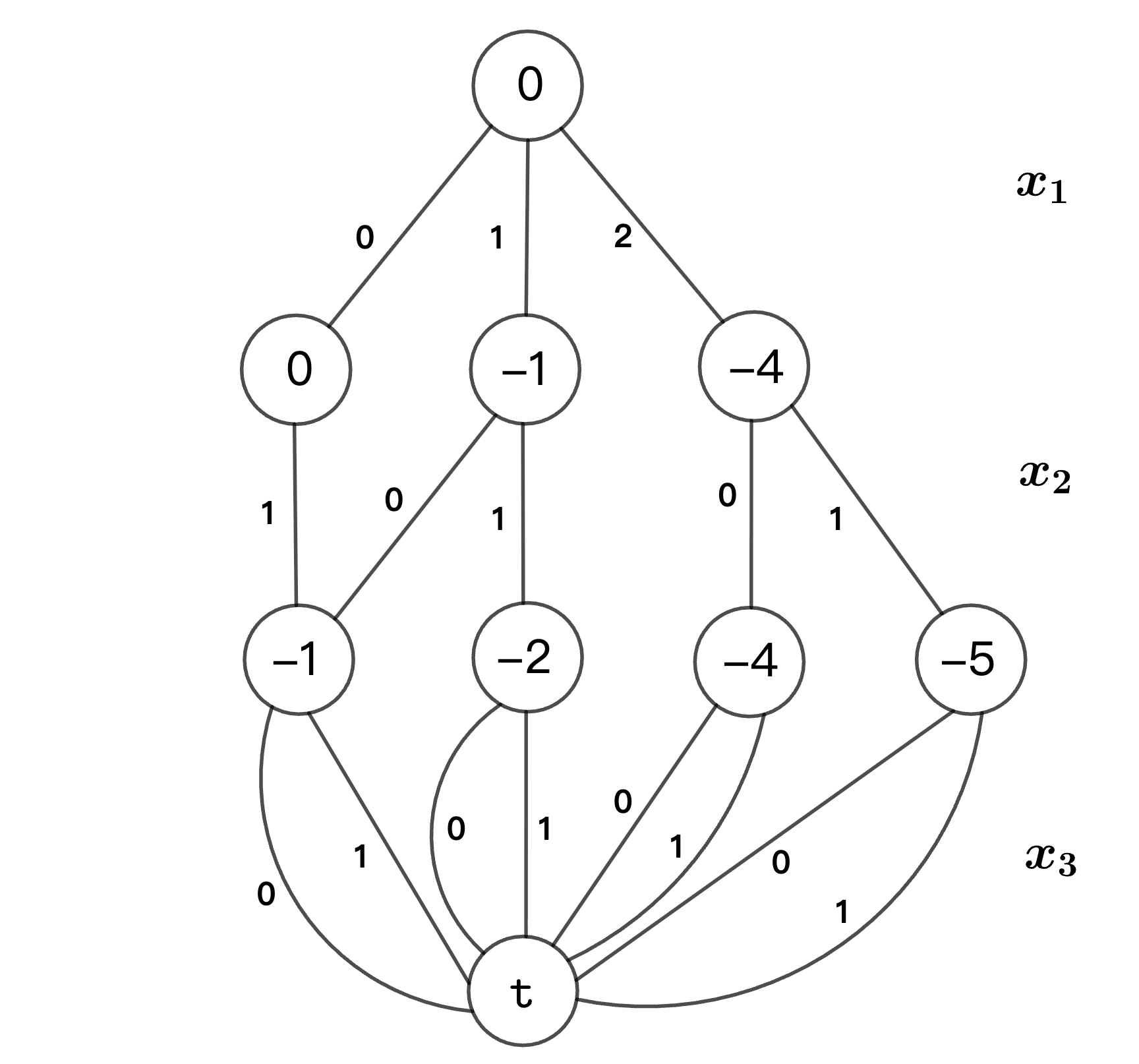}
		\caption{Reduced DD representation}
		\label{fig:nonseparable_2}
	\end{subfigure}
	\caption{Relaxed DD for the set in Example~\ref{ex:nonseparable}.}
	\label{fig:nonseparable}
\end{figure}


According to Algorithm~\ref{alg:nonseparable DD}, the number of nodes at layer $k+1$ for $k \in [n-1]$ of the DD obtained from this algorithm is bounded by $|\mt{U}_k||L_k|$, where $|\mt{U}_k|$ is the number of nodes at layer $k$, and $L_k$ is the number of sub-domain partitions for variable $x_k$.
As a result, the size of this DD grows exponentially as the number of layers increases. 
To control this growth rate, we can use two approaches as outlined below.

\smallskip

The first approach involves controlling the size of the DD by adjusting the number of sub-domain partitions for variables at certain layers.
For instance, assume that there is an imposed width limit of $\omega$ at layer $k+1$, for some $k \in [n]$ in the DD.
To satisfy this width limit at layer $k+1$, we can set the number of sub-domain partitions at layer $k$ to be no greater than $\frac{\omega}{|\mt{U}_k|}$, ensuring that the number of nodes at layer $k+1$ does not exceed $\omega$.

\smallskip

The second approach to controlling the size of the DD involves creating a ``relaxed DD'' by \textit{merging} nodes at layers where the size exceeds the width limit $\omega$.
In this process, multiple nodes in a layer are merged into a single node in such a way that all feasible paths of the original DD are preserved. 
This merging process consists of creating a new node that replaces the merged nodes while inheriting their incoming arcs; see Section~\ref{sub:merging} for detailed steps. 
For the DDs constructed using Algorithm~\ref{alg:nonseparable DD}, setting the state value of the new node as the minimum of state values of the merged nodes ensures that all feasible paths of the original DD are preserved.
Such a merging operation is executed within a \textit{merging oracle} after the nodes in each DD layer are created to ensure that the prescribed width limit is satisfied.
We denote this merging oracle by $\mt{Merge}(\omega,\mt{V})$, where $\omega$ is the width limit, and $\mt{V} = \{\mt{v}_{p_1}, \mt{v}_{p_2}, \dotsc, \mt{v}_{p_t}\}$, for some $t \in \N$, is the set of nodes at a particular layer of the DD.
Applying $\mt{Merge}(\omega,\mt{V})$ would merge all the nodes in $\mt{V}$ into a node $\tilde{\mt{v}}$ with state value $\mt{s}(\tilde{\mt{v}}) = \min_{j=1,\dotsc, t} \{\mt{s}(\mt{v}_{p_j})\}$.
The choice of the subsets $\mt{V}$ whose nodes are to be merged by the merging oracle is determined by a \textit{merging policy}.
In Section~\ref{sub:merging}, we present two general and effective policies that can be utilized within this oracle.

Algorithm~\ref{alg:relaxed nonseparable DD} incorporates a merging oracle after the nodes in each DD layer are created to ensure that the prescribed width limit is satisfied.
\black{The proof of correctness follows a construction similar to that in the proof of Proposition~\ref{prop:nonseparable convex hull}, with the additional step of accounting for the impact of merging on node state values.}
The correctness of this algorithm in providing a relaxation for the convex hull of the solution set described by a non-separable function is established in Theorem~\ref{thm:nonseparable relaxed 2}.
This method is further illustrated in Example~\ref{ex:nonseparable2}.

\begin{algorithm}[!ht]
	\caption{Relaxed DD for a non-separable constraint with merging operation}
	\label{alg:relaxed nonseparable DD}
	\KwData{Set $\mc{G} = \left\{\vc{x} \in \prod_{i=1}^{n} \mc{D}_i \, \middle| \, g(\vc{x}) \leq b \right\}$ with $g(\vc{x}) = \sum_{k=1}^q g_k(\vc{x}_{H_k})$ where $g_k(\vc{x}_{H_k}):\Re^{|H_k|} \to \Re$ is a non-separable function that contains variables with indices in $H_k$, the sub-domain partitions $\mc{D}^j_i$ with $j \in L_i$ for $i \in [n]$, a width limit $\omega$, and a merging oracle $\mt{Merge}(\omega,\mt{V})$}
	\KwResult{A DD $\mt{D}$}
	
	{create the root node $\mt{r}$ in the node layer $\mt{U}_1$ with state value $\mt{s}(\mt{r}) = 0$}\\ 
	{create the terminal node $\mt{t}$ in the node layer $\mt{U}_{n+1}$}
	
	\ForAll{$i \in [n]$}
	{
		\ForAll{$\mt{u} \in \mt{U}_i$, $j \in L_i$}
		{
			\ForAll{$k \in [q]$}
			{
				\If{$i = H_k^{\text{max}}$}
				{
					{calculate $\eta_k$ such that $\eta_k \leq g_k(\vc{x}_{H_k})$ for all $x_i \in \mc{D}^j_i$ and $x_l \in \mc{D}_l(\mt{u})$ for $l \in H_k \setminus \{i\}$}				
				}
				\Else
				{
					calculate $\eta_k = 0$
				}
			}
			
			{calculate $\xi = \mt{s}(\mt{u}) + \sum_{k=1}^q \eta_k$}
			
			\If{$i < n$}
			{
				{create a node $\mt{v}$ with state value $\mt{s}(\mt{v}) = \xi$ (if it does not already exist) in the node layer $\mt{U}_{i+1}$}\\
				{add two arcs from $\mt{u}$ to $\mt{v}$ with label values $\mc{D}^j_i \lb$ and $\mc{D}^j_i \ub$}\\	
				{update $\mc{D}_k(\mt{v})$ for all $k$ such that $k \in H_l$ for some $l > i$}
			}
			\ElseIf{$\xi \leq b$}
			{
				{add two arcs from $\mt{u}$ to the terminal node $\mt{t}$ with label values $\mc{D}^j_i \lb$ and $\mc{D}^j_i \ub$}				
				
			}
		}
		
		\If{$i < n$}
		{
			{merge nodes at this layer to satisfy the prescribed width limit by invoking the merging oracle $\mt{Merge}(\omega,\mt{U}_{i+1})$}	
		}
	}
	
\end{algorithm}

\begin{theorem} \label{thm:nonseparable relaxed 2}
	Consider $\mc{G} = \left\{\vc{x} \in \prod_{i=1}^{n} \mc{D}_i \, \middle| \, g(\vc{x}) \leq b \right\}$ with $g(\vc{x}) = \sum_{j=1}^q g_j(\vc{x}_{H_j})$ where $g_j(\vc{x}_{H_j}):\Re^{|H_j|} \to \Re$ is a non-separable function that contains variables with indices in $H_j \subseteq [n]$. 
	Let $\mt{D}$ be the DD constructed via Algorithm~\ref{alg:relaxed nonseparable DD} for some sub-domain partitions $\mc{D}^j_i$ with $j \in L_i$ for $i \in [n]$, width limit $\omega$, and merging oracle $\mt{Merge}(\omega,\mt{V})$. Then, $\conv(\mathcal{G}) \subseteq \conv(\Sol(\mt{D}))$.\\
\end{theorem}

\begin{proof}
    \black{It is sufficient to prove that $\mathcal{G}\subseteq \conv(\Sol(\mt{D}))$. Fix any $\bar{\vc{x}}\in\mathcal{G}$.}
	It follows from the definition of $\mathcal{G}$ that $\sum_{k=1}^q g_k(\bar{\vc{x}}_{H_k}) \leq b$. 	
	For each $i \in [n]$, let $j^*_i$ be the index of a sub-domain partition $\mc{D}^{j^*_i}_i$ in $L_i$ such that $\bar{x}_i \in \mc{D}^{j^*_i}_i$. 
	This index exists because $\bar{x}_i \in \mc{D}_i = \bigcup_{j \in L_i} \mc{D}^j_i$ where the inclusion follows from the fact that $\bar{\vc{x}} \in \mathcal{G}$, and the equality follows from the definition of sub-domain partitions.	 
	
	\smallskip	
	We show that $\mt{D}$ includes a node sequence $\{\mt{u}_1, \mt{u}_2, \dotsc, \mt{u}_{n+1}\}$, where $\mt{u}_i \in \mt{U}_i$ for $i \in [n+1]$, such that each node $\mt{u}_i$ is connected to $\mt{u}_{i+1}$ via two arcs with labels $\mc{D}^{j^*_i}_i \lb$ and $\mc{D}^{j^*_i}_i \ub$ for each $i \in [n]$.
    \black{The root-to-terminal paths of this node sequence encode the vertices of the hyper-rectangle $\prod_{i=1}^n [\mc{D}_i^{j_i^*}\lb,\mc{D}_i^{j_i^*}\ub]$,
	which contains $\bar{\vc{x}}$ by construction. It then follows that
	$\bar{\vc{x}} \in \conv(\Sol(\mt{D}))$.}
	
	\smallskip
    \black{The argument rests on the following bound on the state value along this sequence.}
	We use induction on layer $i \in [n]$ to prove $\mt{s}(\mt{u}_i) \leq \sum_{k \in [q]: H_k^{\text{max}} < i}g_k(\bar{\vc{x}}_{H_k})$. 
	In words, the state value of each node $\mt{u}_i$ in the previously picked sequence is no greater than the summation of non-separable function terms $g_k(\bar{\vc{x}}_{H_k})$ whose variable have already been visited in the previous layers of the DD.
	The induction base for $i=1$ follows from line 1 of Algorithm~\ref{alg:relaxed nonseparable DD} due to $\mt{s}(\mt{u}_1) = \mt{s}(\mt{r}) = 0$ and the fact that $H_k^{\text{max}} > 0$ for all $k \in [q]$, implying that $\sum_{k \in [q]: H_k^{\text{max}} < i}g_k(\bar{\vc{x}}_{H_k}) = 0$ by default.
	For the inductive hypothesis, assume that $\mt{s}(\mt{u}_{\hat{i}}) \leq \sum_{k \in [q]: H_k^{\text{max}} < \hat{i}}g_k(\bar{\vc{x}}_{H_k})$ for $\hat{i} \in [n-1]$.
	For the inductive step, we show that $\mt{s}(\mt{u}_{\hat{i}+1}) \leq \sum_{k \in [q]: H_k^{\text{max}} < \hat{i}+1}g_k(\bar{\vc{x}}_{H_k})$.
	To calculate $\mt{s}(\mt{u}_{\hat{i}+1})$, we consider two cases for $\mt{u}_{\hat{i}+1}$.
	
	\smallskip
	For the first case, assume that $\mt{u}_{\hat{i}+1}$ is not a merged node created through the merging oracle $\mt{Merge}(\omega,\mt{V})$.
	Therefore, this node must have been created through lines 11--13 of Algorithm~\ref{alg:relaxed nonseparable DD}.
	This case falls into the settings for Proposition~\ref{prop:nonseparable convex hull}. 
	Consequently, \black{the same argument as} in the proof of Proposition~\ref{prop:nonseparable convex hull} \black{yields} that $\mt{s}(\mt{u}_{\hat{i}+1}) \leq \sum_{k \in [q]: H_k^{\text{max}} < \hat{i}+1}g_k(\bar{\vc{x}}_{H_k})$.

	\smallskip
	For the second case, assume that $\mt{u}_{\hat{i}+1}$ is a merged node created through the merging oracle $\mt{Merge}(\omega,\mt{V})$.
	Therefore, before reaching line 18 of Algorithm~\ref{alg:relaxed nonseparable DD}, there must have been a node $\mt{v}$ in layer $\hat{i}+1$ created in lines 11--13 of the algorithm, with arcs connected to $\mt{u}_{\hat{i}}$ with label values $\mc{D}^{j^*_{\hat{i}}}_{\hat{i}} \lb$ and $\mc{D}^{j^*_{\hat{i}}}_{\hat{i}} \ub$, which is subsequently merged into node $\mt{u}_{\hat{i}+1}$ after executing the merging oracle $\mt{Merge}(\omega,\mt{V})$ in line 18 of the algorithm.
	\black{The same argument as in} the first case above \black{yields} that $\mt{s}(\mt{v}) \leq \sum_{k \in [q]: H_k^{\text{max}} < \hat{i}+1}g_k(\bar{\vc{x}}_{H_k})$.
	Since $\mt{v}$ is merged into $\mt{u}_{\hat{i}+1}$, it follows from the merging rule that $\mt{s}(\mt{u}_{\hat{i}+1}) \leq \mt{s}(\mt{v})$, yielding $\mt{s}(\mt{u}_{\hat{i}+1}) \leq \sum_{k \in [q]: H_k^{\text{max}} < \hat{i}+1}g_k(\bar{\vc{x}}_{H_k})$.
	
	\smallskip
	The remainder of the proof follows from arguments similar to those in the proof of Proposition~\ref{prop:nonseparable convex hull}.\\
		\qed       
\end{proof}

{
\color{black}
\begin{example} \label{ex:nonseparable2}
	Consider the MINLP set $\mc{G}$ studied in Example~\ref{ex:nonseparable}.
	Assume that a width limit $\omega = 2$ is imposed.
	Since the DD constructed via Algorithm~\ref{alg:nonseparable DD} in Figure~\ref{fig:nonseparable} does not satisfy this width limit, we employ Algorithm~\ref{alg:relaxed nonseparable DD} to apply a merging operation at layers whose size exceeds the width limit.
	The process of constructing the DD through merging nodes is shown in Figure~\ref{fig:relaxed nonseparable}.
	The first node layer contains the root node with state value $0$.
	After the nodes in the second node layer are created through lines 3--14 of the algorithm, we observe that the number of nodes in this layer exceed the width limit, thereby calling the merging oracle $\mt{Merge}(\omega, \mt{U}_2)$ in line 18.
	According to this oracle, we merge nodes with state values $-1$ and $-4$, as shown in Figure~\ref{fig:relaxed nonseparable_1}.
	The algorithm proceeds with creating the next layer which contains four nodes as shown in Figure~\ref{fig:relaxed nonseparable_2}.
	Similarly to the previous layer, the merging oracle $\mt{Merge}(\omega, \mt{U}_3)$ is called to merge nodes with state values $0$ and $-1$, as well as the nodes with state values $-4$ and $-5$ as depicted in dashed boxes in Figure~\ref{fig:relaxed nonseparable_2}.
	Since the number of nodes in this layer satisfies the width limit, the algorithm continues to the last iteration to create the arcs that are connected to the terminal node.
	Through a calculation similar to that in Example~\ref{ex:nonseparable}, we conclude that the node with state value $-1$ in the node layer $3$ cannot be connected to the terminal node as it does not satisfy the condition in line 15 of Algorithm~\ref{alg:relaxed nonseparable DD}.
	In contrast, the node with state value $-5$ is connected to the terminal node via two arcs with labels $1$ and $2$.
	The resulting DD, which we refer to as $\mt{D}^{\omega}$, is shown in Figure~\ref{fig:relaxed nonseparable_3}.
	To reduce the size of the DD, we can remove the nodes that are not connected to the terminal node together with their incoming arcs to obtain a reduced DD $\mt{\bar{D}}^{\omega}$ shown in Figure~\ref{fig:relaxed nonseparable_4}.
	It is easy to verify that $\conv(\Sol(\mt{\bar{D}}^{\omega})) = [1, 2] \times [0, 1] \times [1, 2]$.
	This set matches the convex hull of the original set $\mc{G}$, i.e., $\conv(\mc{G}) = \conv(\Sol(\mt{\bar{D}}^{\omega}))$.
	$\blacksquare$
\end{example}
}

\begin{figure}[h]
	\centering
	\begin{subfigure}[b]{0.47\linewidth} 
		\centering
		\includegraphics[scale=1,width=\textwidth]{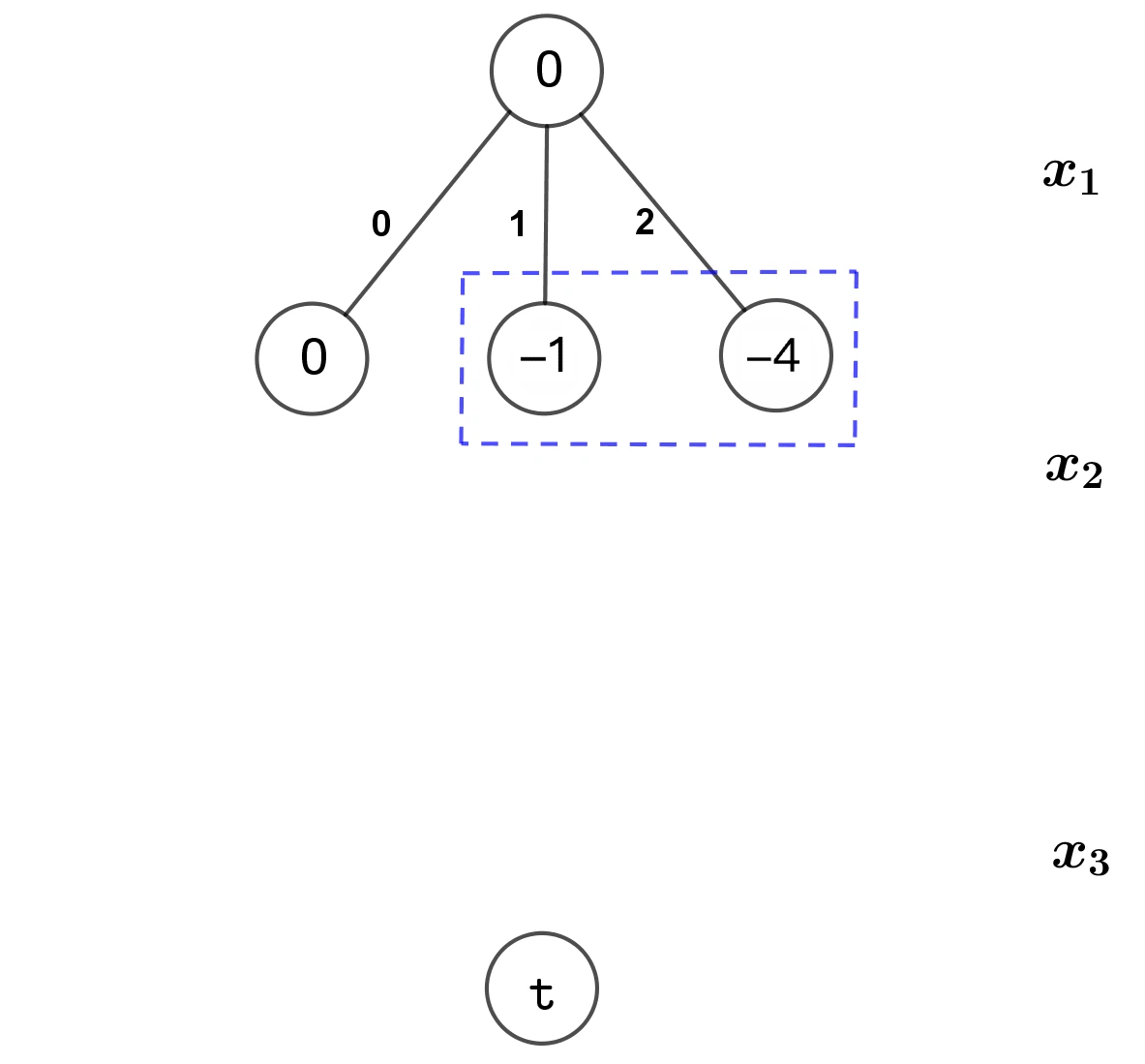}
		\caption{Merging operation at node layer 2}
		\label{fig:relaxed nonseparable_1}
	\end{subfigure}
	\hspace{0.03\linewidth}
	\begin{subfigure}[b]{0.47\linewidth} 
		\centering
		\includegraphics[scale=1,width=\textwidth]{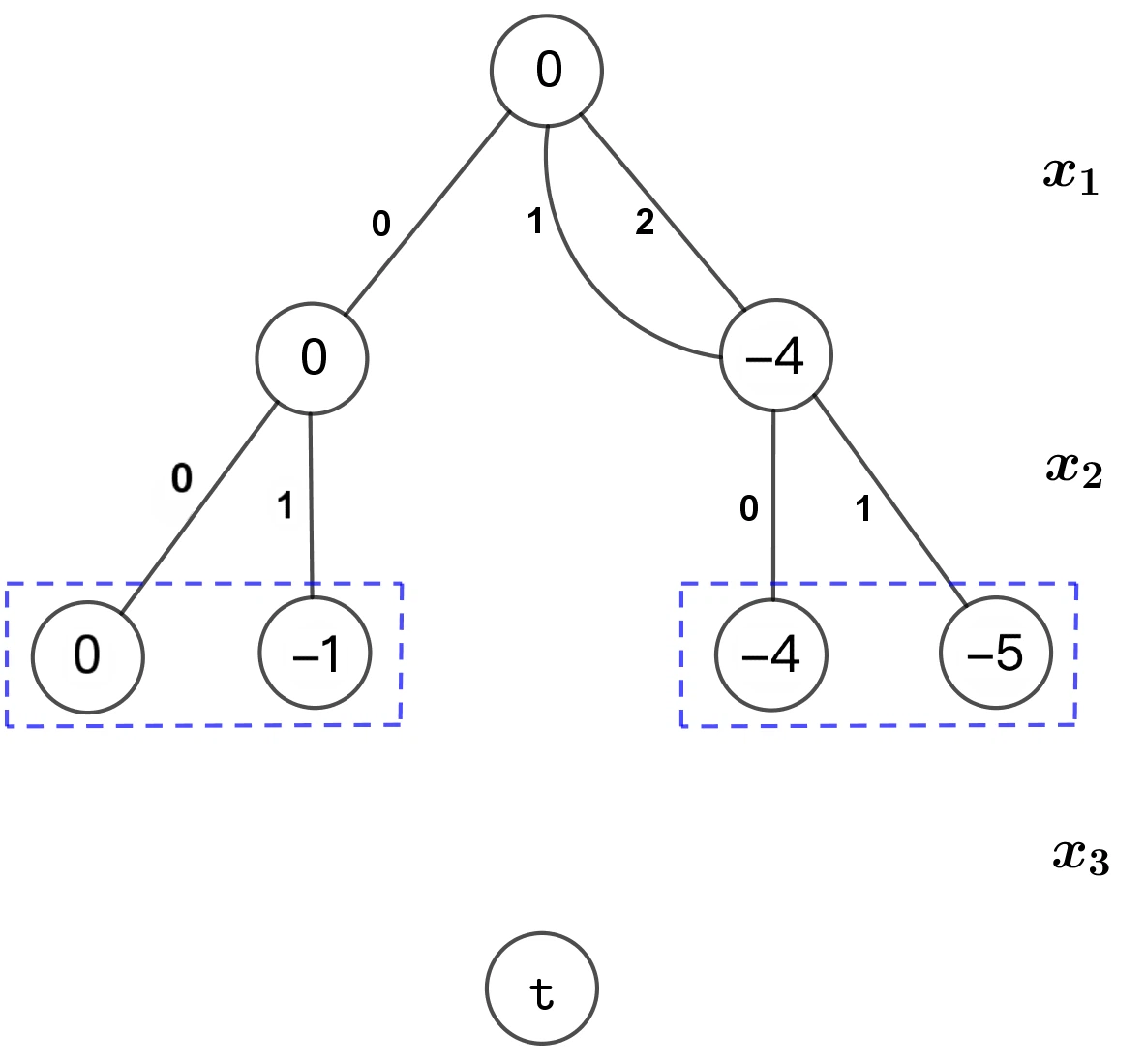}
		\caption{Merging operations at node layer 3}
		\label{fig:relaxed nonseparable_2}
	\end{subfigure}\\
	\begin{subfigure}[b]{0.47\linewidth} 
		\centering
		\includegraphics[scale=1,width=\textwidth]{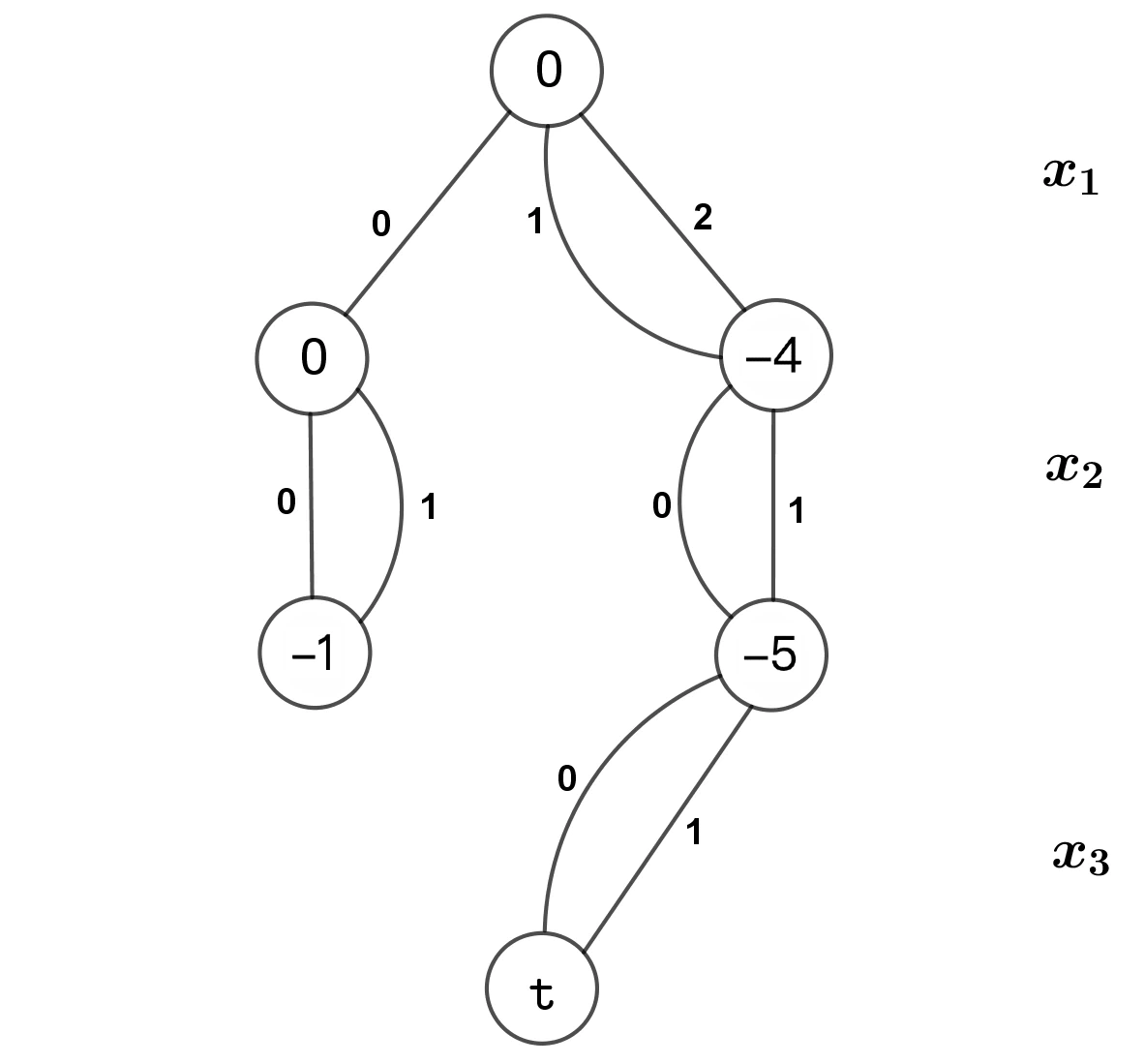}
		\caption{Relaxed DD representation}
		\label{fig:relaxed nonseparable_3}
	\end{subfigure}
	\hspace{0.03\linewidth}
	\begin{subfigure}[b]{0.47\linewidth} 
		\centering
		\includegraphics[scale=1,width=\textwidth]{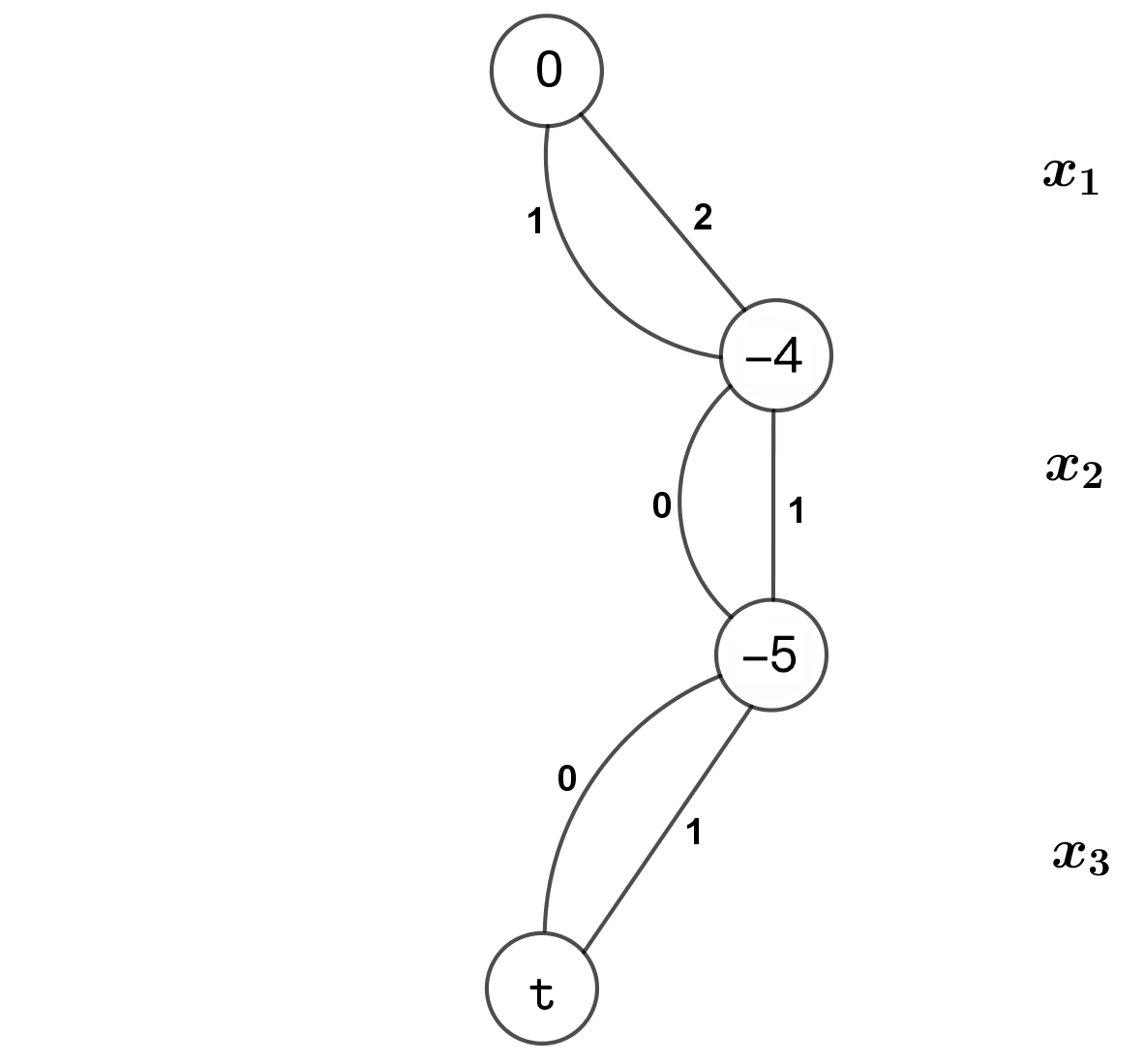}
		\caption{Reduced DD representation}
		\label{fig:relaxed nonseparable_4}
	\end{subfigure}
	\caption{Relaxed DD for the set in Example~\ref{ex:nonseparable2}.}
	\label{fig:relaxed nonseparable}
\end{figure}

Theorem~\ref{thm:nonseparable relaxed 2} implies that applying a merging operation at the layers of a DD that models non-separable functions can effectively reduce the size of the DD while still providing a relaxation for the underlying set.
This result aligns with the well-known findings on relaxed DDs for separable functions in the literature \cite{davarnia:va:2020,davarnia:2021}.
However, there is an interesting and notable difference in how the merging operation impacts the relaxations produced for the separable and non-separable cases.
In the separable case, the merging operation exhibits a \textit{sequential relaxation} property. 
This means that the operation can be decomposed into a sequence of steps, with each step providing a relaxation for the set obtained from the previous step. 
In contrast, the sequential relaxation property does not hold for the non-separable case.
Specifically, merging nodes at the layers of a DD that invoke node backtracking, as in Algorithm~\ref{alg:nonseparable DD}, while still providing a relaxation for the original set, does not necessarily provide a relaxation for the solution set of the DD itself.

\smallskip
This result may initially seem counterintuitive because merging operations are traditionally associated with weakening the relaxation, as they often lead to underestimating the state value of some nodes in the DD. 
However, when dealing with functions that include non-separable terms, the impact of merging operations extends beyond just the state values of individual nodes. 
Merging nodes at a given layer of a DD not only directly affects the state values of those nodes but also indirectly alters the structure of the DD by regrouping the arcs that can reach specific nodes in subsequent layers. 
This restructuring can significantly influence how state values are computed in later layers, potentially leading to stronger relaxations rather than weaker ones.
As a result, the DD obtained after applying merging operations may not necessarily provide a relaxation of the original DD. 
This phenomenon can be observed by comparing the DDs obtained from the non-separable cases in Examples~\ref{ex:nonseparable} and \ref{ex:nonseparable2}.
In particular, we have \black{$\conv(\mt{\bar{D}}) \nsubseteq \conv(\mt{\bar{D}}^{\omega})$}, even though \black{$\mt{\bar{D}}$} is derived by merging some nodes of \black{$\mt{\bar{D}}^{\omega}$}.
Instead, the opposite inclusion holds, i.e., \black{$\conv(\mt{\bar{D}}) \subset \conv(\mt{\bar{D}}^{\omega})$}.
This inversion occurs because the merging operation regrouped the arcs in layer 1 in a manner that strengthened the relaxation of the bilinear term $-x_1x_3$ in layer 3, leading to a tighter overall relaxation.
This example highlights the nuanced effects of merging operations in non-separable DDs and demonstrates that, contrary to what one might expect, these operations can sometimes result in stronger relaxations rather than weaker ones.

\smallskip
The foundation of our DD solution framework lies in constructing outer approximations for the MINLP by convexifying the feasible regions defined by its individual constraints. 
However, a valuable feature that can enhance both the flexibility and strength of our method is the ability to convexify the feasible region defined by an intersection of multiple constraints simultaneously.
The following remark outlines two approaches that can be employed to achieve this feature.

\begin{remark} \label{rem:intersection}	
	Consider sets $\mc{G}^k$ defined in \eqref{eq:OA_constraint} for $k \in K$.	
	The first approach for representing the intersection of multiple constraints via DDs involves the following steps: (i) Begin by constructing a separate DD $\mt{D}^k$ corresponding to each constraint $\mc{G}^k$ for $k \in K$ using Algorithms~\ref{alg:nonseparable DD} or \ref{alg:relaxed nonseparable DD};
	and (ii) intersect the resulting DDs using the well-known ``conjoining'' technique, which yields a DD $\bar{\mt{D}}$ whose feasible solutions consist of the solutions feasible to all individual DDs, i.e., $\Sol(\bar{\mt{D}}) = \bigcap_{k \in K} \Sol(\mt{D}^k)$; see \cite{bergman:ci:va:ho:2016} for a detailed exposition to the conjoining technique.
	
	The second approach involves directly constructing a DD $\tilde{\mt{D}}$ that represents the feasible region defined by the intersected constraints, i.e., $\Sol(\tilde{\mt{D}}) = \bigcap_{k \in K} \mc{G}^k$.
	The construction procedure follows a similar process to that outlined in Algorithms~\ref{alg:nonseparable DD} or \ref{alg:relaxed nonseparable DD}, with a key difference: for each node in the DD, instead of a single state value, there will be a vector of state values. 
	Each component of this vector corresponds to the state values calculated according to the algorithms for each individual constraint $\mc{G}^k$.
	In the final step of these algorithms, a node at layer $n$ is connected to the terminal node if and only if the last if-condition is satisfied for all components of the state value vector. 
\end{remark}

\subsection{Merging Policies}  \label{sub:merging}

In Algorithm~\ref{alg:relaxed nonseparable DD}, the merging oracle $\mt{Merge}(\omega, \mt{U})$ plays a critical role by providing a policy by merging subsets of nodes $\mt{U}$ in a DD layer to satisfy the prescribed DD width $\omega$.
As noted in Example~\ref{ex:nonseparable2} and its subsequent discussions, this merging policy significantly impacts the quality of the relaxation represented by a DD.
Various general-purpose and problem-specific merging policies can be developed to achieve specific properties for the relaxed DD set. 
In this section, we introduce two of the most effective and yet versatile merging oracles that can be applied to any MINLP structures.

\smallskip
The first merging oracle, denoted by $\mt{Merge}^f(.)$, is described in Algorithm~\ref{alg:merge1}.
This algorithm takes as inputs a width limit $\omega$ and a set of nodes $\mt{U}_i = \{\mt{u}_1, \mt{u}_2, \dotsc, \mt{u}_{\kappa}\}$ in a DD layer $i > 1$, where $\kappa = |\mt{U}_i|$. 
The merging policy used in this algorithm merges the $\kappa - \omega + 1$ nodes with the lowest state values to ensure that the width limit is satisfied.
This approach aims to group nodes that are more likely to be part of feasible paths in the DD due to their minimal state value contribution at this layer.
The following result shows that the time complexity of Algorithm~\ref{alg:merge1} depends on the width limit and the number of sub-domain partitions, both of which can be controlled by the user. 

\begin{proposition} \label{prop:merge1}
	Consider a node set $\mt{U}_i$ created in layer $i > 1$ of a DD $\mt{D}$ through Algorithm~\ref{alg:relaxed nonseparable DD}. Let $\omega$ be the width limit and $L_{i-1}$ be the index set of sub-domain partitions at layer $i-1$. Then, Algorithm~\ref{alg:merge1} implements the merging policy defined by $\mt{Merge}^f(\omega, \mt{U}_i)$ with a time complexity of $\mc{O}(\tau \log(\tau))$, where $\tau = \omega|L_{i-1}|$.
\end{proposition}

\begin{proof}
	The runtime complexity of Algorithm~\ref{alg:merge1} is $\mc{O}(|\mt{U}_i| \log(|\mt{U}_i|) + |\mc{A}_{i-1}|)$ since $\mc{O}(|\mt{U}_i| \log(|\mt{U}_i|)$ is the complexity of sorting elements in $\mt{U}_i$ in line 2 of the algorithm, and $|\mc{A}_{i-1}|$ is an upper bound for the number of incoming arcs that are updated in line 5 of the algorithm.
	Since the node set $\mt{U}_i$ is created through Algorithm~\ref{alg:relaxed nonseparable DD}, the number of nodes in this layer can be bounded by the prescribed width limit and the number of sub-domain partitions in that layer.
	In particular, we have $|\mt{U}_i| \leq \omega|L_{i-1}|$.
	This is because of the for-loops in line 4 of Algorithm~\ref{alg:relaxed nonseparable DD}, which imply that for each node $\mt{v} \in \mt{U}_{i-1}$ and each sub-domain partition $j \in L_{i-1}$, one new node can be created in layer $i$ as shown in line 12 of Algorithm~\ref{alg:relaxed nonseparable DD}. 
	Because the previous DD layers satisfy the width limit, we have that $|\mt{U}_{i-1}| \leq \omega$.
	Similarly, the arcs in layer $\mc{A}_{i-1}$ are created within the same for-loops in line 13 of Algorithm~\ref{alg:relaxed nonseparable DD}, yielding $|\mc{A}_{i-1}| \leq 2 \omega |L_{i-1}|$.
	Therefore, we obtain the time complexity $\mc{O}(\tau \log(\tau))$ for the algorithm.  
		\qed       
\end{proof}

\begin{algorithm}[!ht]
	\caption{Merging oracle $\mt{Merge}^f(.)$}
	\label{alg:merge1}
	\KwData{Set $\mt{U}_i$ in a DD layer $i$, a width limit $\omega$}
	\KwResult{A modified set $\mt{U}_i$ that satisfies the width limit}
	
	\If{$|\mt{U}_i| > \omega$}
	{
		{sort the node indices in $\mt{U}_i$ based on their state values, i.e., $\{\mt{u}_{j_1}, \mt{u}_{j_2}, \dotsc, \mt{u}_{j_{\kappa}}\}$ such that $\mt{s}(\mt{u}_{j_1}) \leq \mt{s}(\mt{u}_{j_2}) \leq \dotsc, \mt{s}(\mt{u}_{j_{\kappa}})$} 
		
		{create a node $\mt{v}$ in $\mt{U}_i$ and assign its state value $\mt{s}(\mt{v}) =  \min \big\{ \mt{s}(\mt{u}_{j_1}), \mt{s}(\mt{u}_{j_2}), \dotsc, \mt{s}(\mt{u}_{j_{\kappa - \omega +1}}) \big\}$}
		
		\ForAll{$k \in [\kappa - \omega + 1]$}
		{ 
			
			{disconnect the incoming arcs of $\mt{u}_{j_k}$ from it and connect them to $\mt{v}$}\\
			{delete node $\mt{u}_{j_k}$ from $\mt{U}_i$}
		}			
	}
	
\end{algorithm}

The second merging oracle, denoted by $\mt{Merge}^g(.)$, is given in Algorithm~\ref{alg:merge2}.
Similar to Algorithm~\ref{alg:merge1}, the inputs are the node sets $\mt{U}_i$ and a width limit $\omega$.
This merging oracle divides the entire range of state values of nodes in $\mt{U}_i$ into $\omega$ sub-ranges, then merges all nodes within each sub-range.
The rationale behind this approach is to group nodes with similar state values as merging candidates, thereby reducing the variation in new paths created by merging, hence yielding a tighter relaxation.
The next result shows the time complexity of Algorithm~\ref{alg:merge1}.
The proof is omitted as it follows from arguments similar to those presented in Proposition~\ref{prop:merge1}.

\begin{proposition} \label{prop:merge2}
	Consider a node set $\mt{U}_i$ created in layer $i > 1$ of a DD $\mt{D}$ through Algorithm~\ref{alg:relaxed nonseparable DD}. Let $\omega$ be the width limit and $L_{i-1}$ be the index set of sub-domain partitions at layer $i-1$. Then, Algorithm~\ref{alg:merge2} implements the merging policy defined by $\mt{Merge}^g(\omega, \mt{U}_i)$ with a time complexity of $\mc{O}(\tau \log(\tau))$, where $\tau = \omega|L_{i-1}|$.
\end{proposition}

\begin{algorithm}[!ht]
	\caption{Merging oracle $\mt{Merge}^g(.)$}
	\label{alg:merge2}
	\KwData{Set $\mt{U}_i$ in a DD layer $i$, a width limit $\omega$}
	\KwResult{A modified set $\mt{U}_i$ that satisfies the width limit}
	
	\If{$|\mt{U}_i| > \omega$}
	{
		{sort the node indices in $\mt{U}_i$ based on their state values, i.e., $\mt{U}_s = \{\mt{u}_{j_1}, \mt{u}_{j_2}, \dotsc, \mt{u}_{j_{\kappa}}\}$ such that $\mt{s}(\mt{u}_{j_1}) \leq \mt{s}(\mt{u}_{j_2}) \leq \dotsc, \mt{s}(\mt{u}_{j_{\kappa}})$} 
		
		{calculate the merging interval length $\gamma = \big(\mt{s}(\mt{u}_{j_{\kappa}}) - \mt{s}(\mt{u}_{j_{1}})\big)/\omega$}
		
		\ForAll{$k \in [\omega]$}
		{
			
			{create a node $\mt{v}_k$ in $\mt{U}_i$ and assign its state value $\mt{s}(\mt{v}) =  \min_{l \in \pi} \{ \mt{s}(\mt{u}_l) \}$, where $\pi$ is the set of indices of nodes in $\mt{U}_i$ whose state value lie within $\big[\mt{s}(\mt{u}_{j_1}) + (k-1)\gamma, \, \mt{s}(\mt{u}_{j_1}) + k\gamma \big]$ }
			
			\ForAll{$l \in \pi$}
			{ 
				
				{disconnect the incoming arcs of $\mt{u}_{l}$ from it and connect them to $\mt{v}_k$}\\
				{delete node $\mt{u}_{l}$ from $\mt{U}_i$}
			}
		}			
	}
	
\end{algorithm}

\subsection{Lower Bound Calculation Rules}  \label{sub:lower bound}

A key step in constructing DDs using Algorithms~\ref{alg:nonseparable DD} and \ref{alg:relaxed nonseparable DD} is calculating a lower bound for constraint terms, which is used for determining the state values of each node.
In this section, we discuss the methods that can be employed to calculate these bounds.

\smallskip
The first approach leverages \textit{factorable decomposition}, a widely used technique for obtaining convex relaxations of MINLPs in existing global solvers.
Consider a factorable function $g(\vc{x}):\Re^n \to \Re$, i.e., it can be decomposed into simpler terms with known convex hull representations over box domains.
The goal is to find a lower bound $\eta$ for $g(\vc{x})$ over the box domain described by $x_i \in \mc{D}_i$ for all $i \in [n]$.
Let $\mc{S}$ be the convex relaxation (possibly in a higher dimension) for the set $\big\{(\vc{x},z) \in \prod_{i=1}^n \mc{D}_i \times \Re \, \big| \, z = g(\vc{x})\big\}$ obtained by applying the factorable decomposition technique.
Since $\mc{S}$ may be defined in a higher-dimensional space, we represent its solutions as a vector $(\vc{x}, z, \vc{y}) \in \Re^{n+1+p}$, where components $\vc{y} \in \Re^p$ are auxiliary variables introduced during the factorable decomposition process.
The desired lower bound can then be computed as $\eta = \min \big\{z \big| (\vc{x}, z, \vc{y}) \in \mc{S} \big\}$.
This results in a convex program that can be solved using various existing convex optimization methods.

\smallskip
Although convex programs can be solved in polynomial time, repeatedly invoking a convex solver for a large number of nodes in DD layers can be computationally intensive. Moreover, there are several applications, such as those discussed in Section~\ref{sec:computation}, where the underlying nonlinear functions are not factorable, making them unsuitable for the factorable decomposition approach outlined above. 
Therefore, it is critical to develop an alternative method that can efficiently find lower bounds for a broader class of nonlinear functions when constructing large DDs.
Our proposed method addresses this need by leveraging the monotone properties of the underlying function, as described below.
In the following definition, given a vector $\vc{x} \in \Re^n$, we denote by $\vc{x}_{-i}$ a replica of the vector $\vc{x}$ with component $i \in [n]$ is removed.

\begin{definition} \label{def:monotone}	
	Consider a function $g(\vc{x}):\Re^n \to \Re$ and a box domain described by $x_i \in \mc{D}_i$ for all $i \in [n]$.
	For each $i \in [n]$, we denote by $g(x_i, \bar{\vc{x}}_{-i}):\Re \to \Re$ the univariate restriction of $g(\vc{x})$ in the space of $x_i$ where variables $x_j$ are fixed at $\bar{x}_j$ for all $j \in [n] \setminus \{i\}$.
	We say that $g(x_i, \bar{\vc{x}}_{-i})$ is \textit{monotonically non-decreasing} (resp. \textit{non-increasing}) over $\mc{D}_i$ if $g(\hat{\vc{x}}) \leq g(\tilde{\vc{x}})$ for any pair of points $\hat{\vc{x}}, \tilde{\vc{x}} \in \Re^n$ such that $\hat{x}_j = \tilde{x}_j = \bar{x}_j$ for $j \in [n] \setminus \{i\}$, $\hat{x}_i, \tilde{x}_i \in \mc{D}_i$, and $\hat{x}_i \leq \tilde{x}_i$ (resp. $\hat{x}_i \geq \tilde{x}_i$).
	Further, we say that $g(\vc{x})$ is \textit{monotone} over $\prod_{i=1}^n \mc{D}_i$ if its univariate restriction $g(x_i, \bar{\vc{x}}_{-i})$ is monotonically non-decreasing or non-increasing over $\mc{D}_i$ for each $i \in [n]$ and all fixed values $\bar{x}_j \in \mc{D}_j$ for $j \in [n] \setminus \{i\}$.
\end{definition}

Next, we demonstrate how the monotone property of functions can be used to efficiently find a lower bound by evaluating the function at a specific point within its domain.

\begin{proposition} \label{prop:monotone}
	Consider a monotone function $g(\vc{x}):\Re^n \to \Re$ over a box domain described by $x_i \in \mc{D}_i$ for all $i \in [n]$.
	Then, the minimum value of $g(\vc{x})$ over the above box domain can be calculated as $\eta = g(\tilde{\vc{x}})$ where $\tilde{x}_i = \mc{D}_i \lb$ if $g(x_i, \bar{\vc{x}}_{-i})$ is monotonically non-decreasing, and $\tilde{x}_i = \mc{D}_i \ub$ if $g(x_i, \bar{\vc{x}}_{-i})$ is monotonically non-increasing, for each $i \in [n]$ and all fixed values $\bar{x}_j \in \mc{D}_j$ for $j \in [n] \setminus \{i\}$.
\end{proposition}

\begin{proof}
	Assume by contradiction that $g(\tilde{\vc{x}})$ is not the minimum value of $g(\vc{x})$ over the box domain $\prod_{i=1}^n \mc{D}_i$.
	Then, there exists a point $\hat{\vc{x}} \in \prod_{i=1}^n \mc{D}_i$ such that $g(\hat{\vc{x}}) < g(\tilde{\vc{x}})$.
	Since $\hat{\vc{x}} \neq \tilde{\vc{x}}$, there must exist indexes $i_1, i_2, \dotsc, i_p$ for some $1 \leq p \leq n$ such that $\hat{x}_{i_j} \neq \tilde{x}_{i_j}$ for each $j \in [p]$.
	For each such $j$, construct $\hat{\vc{x}}^{j}$ to be the point where $\hat{x}_k^{j} = \hat{x}_k^{j-1}$ for $k \neq i_j$, and $\hat{x}_{i_j}^{j} = \mc{D}_{i_j} \lb$ if $g(x_{i_j}, \hat{\vc{x}}_{-{i_j}})$ is monotonically non-decreasing over $\mc{D}_{i_j}$, and $\hat{x}_{i_j}^{j} = \mc{D}_{i_j} \ub$ if $g(x_{i_j}, \hat{\vc{x}}_{-{i_j}})$ is monotonically non-increasing over $\mc{D}_{i_j}$.
	In this definition, we set $\hat{\vc{x}}^{0} = \hat{\vc{x}}$.
	Further, the description of $\tilde{\vc{x}}$ in the proposition statement implies that $\hat{\vc{x}}^{p} = \tilde{\vc{x}}$.
	It follows from Definition~\ref{def:monotone} that $g(\hat{\vc{x}}) = g(\hat{\vc{x}}^{0}) \geq g(\hat{\vc{x}}^{1}) \geq \dotsc \geq g(\hat{\vc{x}}^{p}) = g(\tilde{\vc{x}})$.
	This is a contradiction to the initial assumption that $g(\hat{\vc{x}}) < g(\tilde{\vc{x}})$.
		\qed       
\end{proof}

In practice, many nonlinear functions satisfy the monotone property as defined in Definition~\ref{def:monotone}, enabling the use of a fast method, as outlined in Proposition~\ref{prop:monotone}, to calculate lower bounds when constructing DDs for those functions. 
For example, consider a general polynomial function commonly used in MINLP models, defined as $g(\vc{x}) = \prod_{i=1}^n x_i^{\alpha_i}$ with $\alpha_i \in \Re$ for $i \in [n]$.
It is easy to verify that $g(\vc{x})$ is monotone over each orthant.
As another advantage, the monotone property facilitates the calculation of lower bounds for non-factorable functions that are not suitable for factorable decomposition, as demonstrated in the following example.

\begin{example} \label{ex:monotone}
	Consider the $\ell_p$-norm function $g(\vc{x}) = ||\vc{x}||_p = \left(\sum_{i=1}^n x_i^p \right)^{1/p}$, for $p \in (0, \infty)$.
	This function can be convexified using factorable decomposition method over a box domain $\prod_{i=1}^n \mc{D}_i$ in the positive orthant.
	Let $\mc{S}$ be the convex relaxation described by convexifying individual constraints of the following decomposed formulation of the model
	\begin{subequations}
		\begin{align*}
			&z = y_0^{1/p} &\\
			&y_0 = \sum_{i=1}^n y_i &\\
			&y_i = x_i^p &\forall i \in [n]\\
			&x_i \in \mc{D}_i &\forall i \in [n].
		\end{align*}
	\end{subequations}
	According to the previous arguments, minimizing $z$ over $\mc{S}$ provides a lower bound for $g(\vc{x})$ over its box domain.
	Alternatively, we can use Proposition~\ref{prop:monotone} to calculate this lower bound because $g(\vc{x})$ is monotone over its imposed domain.
	An advantage of the latter approach is that it can be executed directly in the space of the original variables, eliminating the need to introduce auxiliary variables and additional constraints to handle decoupled terms, as required by the factorable decomposition approach.
	
	\smallskip
	Now consider the $\ell_0$-norm function $h(\vc{x}) = ||\vc{x}||_0 = \sum_{i=1}^n \mathbb{I}(x_i)$ over the above box domain $\prod_{i=1}^n \mc{D}_i$, where $\mathbb{I}(x_i) = 0$ if $x_i = 0$, and $\mathbb{I}(x_i) = 1$ otherwise.
	This function is not factorable, making it not unsuitable for application of factorable decomposition.
	In contrast, it is easy to verify that $h(\vc{x})$ is monotone over its box domain.
	Therefore, we can use Proposition~\ref{prop:monotone} to calculate its lower bound.
	$\blacksquare$
\end{example}

Despite the wide range of functions that satisfy the monotone property, some functions do not exhibit this property due to the presence of variables in multiple positions, breaking the monotonic patterns of their univariate restrictions. 
For instance, consider the function $g(x_1, x_2) = \frac{x_2^4 \, e^{-x_2}}{\arctan(x_1) + 1}$ over the positive orthant.
It is clear that the univariate restriction of $g(x_1, x_2)$ in the space of $x_1$, i.e., $g(x_1, \bar{x}_{-1}) = \frac{\bar{x}_2^4 \, e^{-\bar{x}_2}}{\arctan(x_1) + 1}$, is monotonically non-increasing over this domain.
However, the univariate restriction of $g(x_1, x_2)$ in the space of $x_2$, i.e., $g(x_2, \bar{x}_{-2}) = \frac{x_2^4 \, e^{-x_2}}{\arctan(\bar{x}_1) + 1}$, is not monotone.
Consequently, the method of Proposition~\ref{prop:monotone} cannot be used to find a lower bound for $g(x_1, x_2)$.
To address such function structures, we introduce a technique referred to as \textit{re-indexing}, which is outlined next.

\begin{definition} \label{def:reindex}	
	Consider a function $g(\vc{x}):\Re^n \to \Re$.
	Define the \textit{re-indexed} function $g^{\text{rx}}(\vc{y})$ of $g(\vc{x})$ by substituting the variables $\vc{x}$ with variables $\vc{y}$ such that each $y_i$ variable appears only once in the function's expression.
	If variable $y_{j}$ substitutes variable $x_i$ in the re-indexed function, we denote the relation between these indices by the mapping $R(j) = i$.
\end{definition}

In the example discussed previously, the re-indexed function of $g(x_1, x_2)$ is $g^{\text{rx}}(y_1, y_2, y_3) = \frac{y_1^4 \, e^{-y_2}}{\arctan(y_3) + 1}$, where $R(1) = 2$, $R(2) = 2$, and $R(3) = 1$.
It is easy to verify that $g^{\text{rx}}(y_1, y_2, y_3)$ is monotone over the positive orthant, allowing the application of Proposition~\ref{prop:monotone} to calculate its lower bound.
The next proposition shows that this lower bound can also serve as a lower bound for the original function $g(x_1, x_2)$.

\begin{proposition} \label{prop:reindexing}
	Consider a function $g(\vc{x}):\Re^n \to \Re$ over a box domain described by $x_i \in \mc{D}_i$ for all $i \in [n]$.
	Let $g^{\text{rx}}(\vc{y}):\Re^{p} \to \Re$ be the re-indexed function of $g(\vc{x})$ with re-index mapping $R(.)$.
	Assume that $g^{\text{rx}}(\vc{y})$ is monotone over the box domain described by $y_j \in \mc{D}_{R(j)}$ for all $j \in [p]$.
	Then, a lower bound for $g(\vc{x})$ over the above box domain can be calculated as $\eta = g^{\text{rx}}(\tilde{\vc{y}})$ where $\tilde{y}_j = \mc{D}_{R(j)} \lb$ if $g^{\text{rx}}(y_j, \bar{\vc{y}}_{-j})$ is monotonically non-decreasing, and $\tilde{y}_j = \mc{D}_{R(j)} \ub$ if $g^{\text{rx}}(y_j, \bar{\vc{y}}_{-j})$ is monotonically non-increasing, for each $j \in [p]$ and all fixed values $\bar{y}_k \in \mc{D}_{R(k)}$ for $k \in [p] \setminus \{j\}$.
\end{proposition}

\begin{proof}
	Since $g^{\text{rx}}(\vc{y})$ is monotone over the box domain described by $y_j \in \mc{D}_{R(j)}$ for all $j \in [p]$, Proposition~\ref{prop:monotone} implies that $\eta$ is the minimum of $g^{\text{rx}}(\vc{y})$ over this box domain, i.e., $\eta = \min \big\{ g^{\text{rx}}(\vc{y}) \, \big| \,  \vc{y} \in \prod_{j=1}^p \mc{D}_{R(j)} \big\}$.
	Now consider point $\vc{x}^* \in \prod_{i=1}^n \mc{D}_i$ that achieves the minimum of $g(\vc{x})$ over the box domain $\prod_{i=1}^n \mc{D}_i$, i.e., $g(\vc{x}^*) \leq g(\vc{x})$ for all $\vc{x} \in \prod_{i=1}^n \mc{D}_i$.
	Construct the point $\vc{y}^* \in \Re^n$ such that $y^*_j = x^*_{R(j)}$ for all $j \in [p]$.
	Therefore, for each $\vc{x} \in \prod_{i=1}^n \mc{D}_i$, we can write that $g(\vc{x}) \geq g(\vc{x}^*) = g^{\text{rx}}(\vc{y}^*) \geq \eta$, where the first inequality holds because of the previous argument, the first equality follows from the definition of $g^{\text{rx}}(\vc{y})$, and the second inequality is due to the first argument in the proof.
	As a result, $\eta$ is a lower bound for $g(\vc{x})$ over the above box domain $\prod_{i=1}^n \mc{D}_i$.\\
		\qed       
\end{proof}

From a practical standpoint, the combination of exploiting the monotone property of functions as outlined in Proposition~\ref{prop:monotone}, the re-indexing technique described in Proposition~\ref{prop:reindexing}, and the DD intersection method discussed in Remark~\ref{rem:intersection} offers a unique and powerful modeling tool for constructing relaxed DDs across a wide range of MINLP structures. 
In fact, our observations suggest that these techniques can handle most structures in the MINLP library, including the most complex types that remain unsolved due to their inadmissibility by state-of-the-art global solvers; see the computational studies in Section~\ref{sec:computation}. 

\black{It is worth noting that certain problem structures are prone to weaker DD-based relaxations and, consequently, slower convergence to global optimality.
For example, re-indexing may naturally lead to weaker lower bounds as the dimension of the re-indexed space increases, since the enlarged box in the re-indexed space does not enforce the equality relations among variables corresponding to the same original variable.
Furthermore, for instances with many non-separable terms involving large subsets of variables, the resulting backtracking-based subdomain information may become relatively coarse, potentially weakening the relaxation.}

\subsection{Time Complexity of Algorithms}  \label{sub:complexity}

In this section, we analyze the time complexity of Algorithm~\ref{alg:relaxed nonseparable DD}, incorporating the merging policies outlined in Section~\ref{sub:merging} and lower bound calculation rules discussed in Section~\ref{sub:lower bound}. 
Specifically, we assume that the lower bounds on the functions are calculated using the monotone property and the re-indexing technique described in Section~\ref{sub:lower bound}.
Proposition~\ref{prop:complexity relaxed nonseparable} provides the time complexity results for Algorithm~\ref{alg:relaxed nonseparable DD} that builds a DD corresponding to constraints with non-separable functions, as developed in Section~\ref{sub:nonseparable}.

\begin{proposition} \label{prop:complexity relaxed nonseparable}
	Consider set $\mc{G} = \left\{\vc{x} \in \prod_{i=1}^{n} \mc{D}_i \, \middle| \, g(\vc{x}) \leq b \right\}$ where $g(\vc{x}) = \sum_{k=1}^q g_k(\vc{x}_{H_k})$ such that $g_k(\vc{x}_{H_k}):\Re^{|H_k|} \to \Re$ is a non-separable function that contains variables with indices in $H_k$. 
	Consider the sub-domain partitions $\mc{D}^j_i$ with $j \in L_i$ for $i \in [n]$, a width limit $\omega$, and a merging oracle $\mt{Merge}(\omega,\mt{V})$ described in Algorithm~\ref{alg:merge1} or \ref{alg:merge2}. 
	Then, Algorithm~\ref{alg:relaxed nonseparable DD} constructs a DD $\mt{D} = (\mt{U}, \mt{A}, \mt{l}(.))$ corresponding to $\mc{G}$ in time $\mc{O}\big(\sum_{i=1}^n\tau_i\log(\tau_i) + \sum_{k = 1}^q \theta_k\big)$, where $\tau_i = \omega|L_i|$, and $\theta_k = \omega|H_k| \sum_{l=H_k^{\text{min}}}^{H_k^{\text{max}}-1} |L_l|$ with $H_k^{\text{max}} = \max_{l \in H_k}\{l\}$ and $H_k^{\text{min}} = \min_{l \in H_k}\{l\}$.
\end{proposition}	

\begin{proof}
	It follows from the for-loops in lines 3--5 of Algorithm~\ref{alg:relaxed nonseparable DD} that at each layer $i \in [n]$, the lower bound calculation in lines 6--10 of the algorithm, as well as the node and arc creation in lines 12--16 of the algorithm can be performed in $\mc{O}(|\mt{U}|_i |L_i|)$.
	Considering that $|\mt{U}_i|$ is bounded by the width limit $\omega$, we obtain the time complexity of $\mc{O}(\omega|L_i|)$ for the above operations.
    \black{On the other hand, Proposition~\ref{prop:merge1} and \ref{prop:merge2} imply that the merging operation in line 18 of Algorithm~\ref{alg:relaxed nonseparable DD} can be performed in $\mc{O}\big(\omega|L_i|\log(\omega|L_i|)\big) = \mc{O}(\tau_i \log(\tau_i))$ for each $i \in [n]$.} 
	Considering all layers, we obtain the total time complexity for these tasks to be $\mc{O}(\sum_{i=1}^n\tau_i\log(\tau_i))$.
	
	\smallskip
	Next, we obtain the time complexity for calculating the relative sub-domains in line 14 of the algorithm.
	Corollary~\ref{cor:relative domain complexity} implies that the sub-domain $\mc{D}_j(\mt{v})$ relative to variable $x_j$ for all nodes $\mt{v}$ in layer $i \in [n]$ can be computed in $\mc{O}(\sum_{l=j}^{i-1} |\mt{A}_l|)$.
	Using a similar argument to that given previously, we can bound the above term by $\mc{O}(\omega \sum_{l=j}^{i-1} |L_l|)$.	  
	These values are calculated for each layer $i = H_k^{\text{max}}$ for all $k \in [q]$.
	Furthermore, for the nodes $\mt{v}$ in layer $i = H_k^{\text{max}}$ for each $k \in [q]$, we need to calculate the sub-domain $\mc{D}_j(\mt{v})$ for all $j \in H_k \setminus \{i\}$.
	Therefore, for a given $k \in [q]$, the time complexity for calculating the sub-domains relative to all $x_j$ with $j \in H_k \setminus \{H_k^{\text{max}}\}$ can be bounded by $\mc{O}\big(\omega|H_k| \sum_{l=H_k^{\text{min}}}^{H_k^{\text{max}}-1} |L_l|\big) = \mc{O}(\theta_k)$.
	This yields the total time complexity for calculating the relative sub-domains in line 14 of the algorithm to be bounded by $\mc{O}\big(\sum_{k =1}^q \theta_k \big)$.\\
		\qed       
\end{proof}

\section{Outer Approximation} \label{sec:OA}

In this section, we describe the oracle $\mt{Outer\_Approx}$ in Algorithm~\ref{alg:global}.
This oracle produces a linear outer approximation for the solutions of the DD constructed by $\mt{Construct\_DD}$ to find \black{lower bounds}.
Recently, \cite{davarnia:2021,davarnia:va:2020} proposed efficient methods to obtain a convex hull description for the solution set of DDs in the original space of variables through a successive generation of cutting planes.
In this section, we present a summary of those methods, adapted for the DDs constructed in Section~\ref{sec:construction}; refer to the references above for detailed derivations.
We begin by describing the convex hull in an extended space of variables.

\begin{proposition} \label{prop:Behle}
	Consider a DD $\mt{D} = (\mt{U},\mt{A},\mt{l}(.))$ with solution set $\Sol(\mt{D}) \subseteq \Re^n$.
	Define\\
	$\mc{P} = \left\{(\vc{x};\vc{y}) \in \Re^n \times \Re^{|\mt{A}|} \middle| \eqref{eq:NM1},\eqref{eq:NM2} \right\}$ where
	\begin{subequations}
		\begin{align}
			&\sum_{\mt{a} \in \delta^+(\mt{u})} y_{\mt{a}} - \sum_{\mt{a} \in \delta^-(\mt{u})} y_{\mt{a}} = f_{\mt{u}}, & \forall \mt{u} \in \mt{U} \label{eq:NM1}\\
			&\sum_{\mt{a}\in \mt{A}_i} \mt{l}(\mt{a}) \, y_{\mt{a}} = x_i, & \forall i \in [n] \label{eq:NM2}\\
			&y_{\mt{a}} \geq 0,	& \forall \mt{u} \in \mt{U}, \label{eq:NM3}
		\end{align}
	\end{subequations}
	where $f_{\mt{r}} = -f_{\mt{t}} = 1$, $f_{\mt{u}} = 0$ for $\mt{u} \in \mt{U} \setminus \{\mt{r},\mt{t}\}$, and $\delta^+(\mt{u})$ (\textit{resp.} $\delta^-(\mt{u})$) denotes the set of outgoing (\textit{resp.} incoming) arcs at node $\mt{u}$.
	Then, $\proj_{x}\mathcal{P} = \conv(\Sol(\mt{D}))$.
		\qed
\end{proposition}

Viewing $y_{\mt{a}}$ as the network flow variable on arc $\mt{a} \in \mt{A}$ of $\mt{D}$, the formulation \eqref{eq:NM1}--\eqref{eq:NM3} implies that the LP relaxation of the network model that routes one unit of supply from the root node to the terminal node of the DD provides a convex hull description for the solution set of $\mt{D}$ in a higher dimension.
Thus, projecting out the arc-flow variables $\vc{y}$ from this formulation would yield $\conv(\Sol(\mt{D}))$ in the original space of variables.
This result leads to a separation oracle that can be used to separate any point $\bar{\vc{x}} \in \Re^n$ from $\conv(\Sol(\mt{D}))$ through solving the cut-generating LP given in Proposition~\ref{prop:separation} below.
In this model, $\vc{\theta} \in \Re^{|\mt{U}|}$ and $\vc{\gamma} \in \Re^{n}$ are dual variables associated with constraints \eqref{eq:NM1} and \eqref{eq:NM2}, respectively.

\begin{proposition} \label{prop:separation}
	Consider a DD $\mt{D} = (\mt{U},\mt{A},\mt{l}(.))$ with solution set $\Sol(\mt{D}) \subseteq \Re^n$.	
	Consider a point $\bar{\vc{x}} \in \Re^n$, and define
	\begin{align}
		\omega^* = \max \quad & \sum_{i\in [n]}\bar{x}_i\gamma_i - \theta_{\mt{t}} \label{eq:proj-cone0} \\
		&\theta_{\mt{t}(\mt{a})}-\theta_{\mt{h}(\mt{a})} + \mt{l}(\mt{a}) \gamma_i \leq 0, &\forall i \in [n], \mt{a} \in \mt{A}_k \label{eq:proj-cone1}\\
		&\theta_{\mt{r}} = 0. \label{eq:proj-cone2}
	\end{align}	
	Then, $\bar{\vc{x}} \in \conv(\Sol(\mt{D}))$ if $\omega^* = 0$. 
	Otherwise, $\bar{\vc{x}}$ can be separated from $\conv(\Sol(\mt{D}))$ via $\sum_{i\in [n]} x_i\gamma^*_i \leq \theta^*_{\mt{t}}$ where $(\vc{\theta}^*;\vc{\gamma}^*)$ is an optimal recession ray of \eqref{eq:proj-cone0}--\eqref{eq:proj-cone2}. 
		\qed
\end{proposition}

The above separation oracle requires solving an LP whose size is proportional to the number of nodes and arcs of the DD, which could be computationally intensive when used repeatedly inside an outer approximation framework.
As a result, an alternative subgradient-type method is proposed to solve the same separation problem, but with a focus on detecting a violated cut faster.

\begin{algorithm}[!ht]
	\caption{A subgradient-type separation algorithm}          
	\label{alg:subgradient}			                        
	
	\KwData{A DD $\mt{D} = (\mt{U},\mt{A},\mt{l}(.))$ and a point $\bar{\vc{x}}$}
	\KwResult{A valid inequality to separate $\bar{\vc{x}}$ from $\conv(\Sol(\mt{D}))$}
	{initialize $\tau = 0$, $\vc{\gamma}^{0} \in \Re^n$, $\tau^* = 0$, $\Delta^* = 0$}\\
	\While{$\text{Terminate\_Flag} = \text{False}$}
	{
		assing weights $\mt{w}(\mt{a}) = \mt{l}(\mt{a})\gamma^{\tau}_i$ to each arc $\mt{a} \in \mt{A}_i$ of $\mt{D}$ for all $i \in [n]$\\
		find a longest $\mt{r}$-$\mt{t}$ path in the weighted DD and compute its encoding point $\vc{x}^{\tau}$\\ 
		\If{$\vc{\gamma}^{\tau} (\bar{\vc{x}} - \vc{x}^{\tau}) > \max\{0, \Delta^*\}$}
		{update $\tau^* = \tau$ and $\Delta^* = \vc{\gamma}^{\tau} (\bar{\vc{x}} - \vc{x}^{\tau})$} 
		update $\vc{\phi}^{\tau+1} = \vc{\gamma}^{\tau} + \rho_{\tau} (\bar{\vc{x}} - \vc{x}^{\tau})$ for step size $\rho_{\tau}$\\
		find the projection $\vc{\gamma}^{\tau+1}$ of $\vc{\phi}^{\tau+1}$ onto the unit sphere defined by $||\vc{\gamma}||_2 \leq 1$\\  
		set $\tau = \tau + 1$
	}
	\If{$\Delta^* > 0$}
	{return inequality $\vc{\gamma}^{\tau^*} (\vc{x} - \vc{x}^{\tau^*}) \leq 0$}
\end{algorithm}

In Algorithm~\ref{alg:subgradient}, $\textit{Terminate\_Flag}$ contains criteria to stop the loop, such as iteration number, elapsed time, objective function improvement tolerance, among others.
We summarize the recursive step of the separation method employed in this algorithm as follows.
The vector $\vc{\gamma}^{\tau} \in \Re^n$ is used in line 3 to assign weights to the arcs of the DD, which are then used to obtain the longest $\mt{r}$-$\mt{t}$ path.
The solution $\vc{x}^{\tau}$ corresponding to this longest path is subtracted from the separation point $\bar{\vc{x}}$, yielding the subgradient value for the objective function of the separation problem at the point $\vc{\gamma}^{\tau}$; see Proposition 3.5 in \cite{davarnia:va:2020}.
The subgradient direction is then updated in line 7 for a step size $\rho_{\tau}$, and subsequently projected onto the unit sphere of the variables $\vc{\gamma}$ in line 8.
It is shown in \cite{davarnia:va:2020} that for an appropriate step size, this algorithm converges to an optimal recession ray of the separation problem \eqref{eq:proj-cone0}--\eqref{eq:proj-cone2}, thereby producing the desired cutting plane in line 11.
This algorithm is derivative-free, as it computes subgradient values by solving a longest path problem over a weighted DD.
Consequently, it is highly effective in identifying violated cutting planes compared to the LP \eqref{eq:proj-cone0}--\eqref{eq:proj-cone2}, making it well-suited for implementation within the spatial branch-and-cut framework used in Algorithm~\ref{alg:global}.

\smallskip
The cutting planes obtained from the separation methods in Proposition~\ref{prop:separation} and Algorithm~\ref{alg:subgradient} can be incorporated into $\mt{Outer\_Approx}$ as follows.
In the recursive steps of Algorithm~\ref{alg:global}, the LP relaxation $LP$ at a node of the \BB tree is solved to obtain an optimal solution $\vc{x}^*$, if one exists.
For each constraint $k \in K$ in the MINLP \eqref{eq:MINLP-1}--\eqref{eq:MINLP-3}, the solution $\vc{x}^*$ is evaluated to identify any violated constraints.
For each violated constraint, the aforementioned separation methods are employed to generate a cutting plane that separates $\vc{x}^*$ from $\conv(\Sol(\mt{D}^k))$, where $\mt{D}^k$ is the DD constructed for set $\mc{G}^k$ in line 7 of Algorithm~\ref{alg:global}.
The resulting cutting plane is then added to the LP relaxation, and the process is repeated until no new cuts are introduced or a stopping criterion, such as a maximum number of iterations or gap tolerance, is met.
Subsequently, the bounds are updated, and if the current node is not pruned, a spatial branch-and-bound scheme is applied, as discussed in the following section.

\section{Spatial Branch-and-Bound} \label{sec:SBB}

In global optimization of MINLPs, a divide-and-conquer strategy, such as spatial branch-and-bound (SB\&B), is employed to achieve convergence to a global optimal solution of the problem.
The \SBB strategy reduces the domain of the variables by successively partitioning their original box domains.
These partitions are typically rectangular, dividing the variable domain into smaller hyper-rectangles as a result of branching.
For each such partition, a convex relaxation is constructed to calculate a dual bound.
As the process advances, tighter relaxations are obtained, leading to updated dual bounds, which continue to improve until they approach the global optimal value of the problem within a specified tolerance.
To establish convergence, it must be shown that the convexification method applied at each partition converges (in the Hausdorff sense) to the convex hull of the feasible region restricted to that partition; refer to \cite{belotti:le:li:ma:wa:2009,ryoo:sa:1996,tawarmalani:sahinidis:2004} for a detailed discussion on \SBB methods for MINLPs.

\smallskip
In this section, we discuss the convergence results for the \SBB procedure employed in Algorithm~\ref{alg:global}.
After solving a linear outer approximation of the MINLP at the current node of the \SBB tree and updating the bounds, the $\mt{Branch}$ oracle in line 13 of the algorithm is invoked to perform the branching operation, provided the node is not pruned. 
This operation creates two child nodes by partitioning the domain of the selected branching variable based on the branching value.
In the sequel, we show that the convex hull of the solution set of the DDs obtained from Algorithms~\ref{alg:nonseparable DD}--\ref{alg:relaxed nonseparable DD} converges to the convex hull of the solutions of the original set $\mt{G}$ as the partition volume decreases.
\black{This convergence result holds for any exhaustive rectangular branching scheme. That is, along any infinite branch of the \SBB tree, the generated boxes form a nested exhaustive sequence whose diagonal lengths tend to zero.}
\black{Without loss of generality, we establish these results for DDs of unit width. This is because any DD with a larger width can be decomposed into finitely many unit-width DDs by fixing one node sequence from the root to the terminal and retaining all arcs between consecutive nodes in that sequence. Thus, the convergence result for unit-width DDs extends directly to DDs of arbitrary finite width.}
Throughout this section, we assume that the domain partitioning performed through \SBB takes into account the integrality requirements for integer variables.
For instance, if an integer variable $x$ within the domain $[l, u]$, where $l, u \in \Z$ and $l < u$, is selected for branching at a value $w \in [l, u]$, the new domain partitions will be $[l, \lfloor w \rfloor ]$ and $[\lfloor w \rfloor + 1, u]$.

\smallskip
First, we prove that reducing the variables' domain through \SBB partitioning leads to tighter convex relaxations obtained by the proposed DD-based convexification method described in Sections~\ref{sec:construction} and \ref{sec:OA}.
Proposition~\ref{prop:spatial BB nonseparable} establishes this result for the general non-separable constraints. 
To prove this, we rely on a key property of the lower bound calculation rules used in the \SBB process, which we define next.

\begin{definition} \label{def:consistent lower bound}
	Consider a function $g(\vc{x}):\mc{D} \to \Re$, where $C \subseteq [n]$ and $I = [n] \setminus C$ represent the index sets of continuous and integer variables, respectively, and where $\mc{D} = \prod_{i=1}^n \mc{D}_i$ with $\mc{D}_i = [\mc{D}_i \lb, \mc{D}_i \ub]$ for $i \in C$ and $\mc{D}_i = [\mc{D}_i \lb, \mc{D}_i \ub] \cap \Z$ for $i \in I$.
	Consider a lower bound calculation rule that outputs a lower bound $\eta(\bar{\mc{D}})$ for $g(\vc{x})$ over a box domain $\bar{\mc{D}} \subseteq \mc{D}$.
	We say that this lower bound calculation rule is \textit{consistent with respect to} $g(\vc{x})$ \textit{over} $\mc{D}$ if $\eta(\mc{D}^1) \geq \eta(\mc{D}^2)$ for any $\mc{D}^1 \subseteq \mc{D}^2 \subseteq \mc{D}$.
\end{definition}


\begin{proposition}	\label{prop:spatial BB nonseparable}
	Consider a function $g(\vc{x}):\mc{P} \to \Re$, where $g(\vc{x}) = \sum_{j=1}^q g_j(\vc{x}_{H_j})$, with each $g_j(\vc{x}_{H_j})$ being a non-separable function that contains variables with indices in $H_j \subseteq [n]$.
	Let $C \subseteq [n]$ and $I = [n] \setminus C$ represent the index sets of continuous and integer variables, respectively.
	Consider $\mc{P} = \prod_{i=1}^n \mc{P}_i$, where $\mc{P}_i = [\mc{P}_i \lb, \mc{P}_i \ub]$ for $i \in C$ and $\mc{P}_i = [\mc{P}_i \lb, \mc{P}_i \ub] \cap \Z$ for $i \in I$.
	For each $k =1, 2$, let $\mt{D}^k$ be the DD constructed via Algorithm~\ref{alg:nonseparable DD} or \ref{alg:relaxed nonseparable DD} for a single sub-domain partition $\mc{P}^k_i$ of variable $x_i$ for each $i \in [n]$ using a lower bound calculation rule consistent with respect to $g_k(\vc{x}_{H_k})$ over $\prod_{l \in H_k}\mc{P}_l$ for each $k \in [q]$. If $\mc{P}^2 \subseteq \mc{P}^1$, then $\conv(\Sol(\mt{D}^2)) \subseteq \conv(\Sol(\mt{D}^{1}))$.	
\end{proposition}

\begin{proof}
	Since there is only one sub-domain partition for each variable, the DDs constructed via Algorithm~\ref{alg:nonseparable DD} and \ref{alg:relaxed nonseparable DD} are the same.
	Thus, we show the result assuming Algorithm~\ref{alg:nonseparable DD} is used.
	According to this algorithm, because $\mt{D}^2$ has a unit width, we denote by $\mt{u}_i$ the only node at each node layer $i \in [n]$ of this DD. 
	Following the top-down construction steps of Algorithm~\ref{alg:nonseparable DD}, for each $i \in [n-1]$, $\mt{u}_i$ is connected via two arcs with label values $\mc{P}_i^2 \lb$ and $\mc{P}_i^2 \ub$ to $\mt{u}_{i+1}$. 
	For layer $i=n$, we refer to the $\xi$ value computed in line 9 of this algorithm as $\xi^*$ to distinguish it from the values calculated at the previous layers.
	There are two cases for $\xi^*$.	
	
	\smallskip
	For the first case, assume that $\xi^* > b$. Then, the if-condition in line 14 of Algorithm~\ref{alg:nonseparable DD} is not satisfied. 
	Therefore, node $\mt{u}_{n}$ is not connected to the terminal node $\mt{t}$ of $\mt{D}^{2}$. 
	As a result, there is no $\mt{r}$-$\mt{t}$ path in this DD, leading to an empty solution set, i.e., $\conv(\Sol(\mt{D}^{2})) = \Sol(\mt{D}^{2}) = \emptyset$. 
	This proves the result since $\emptyset \subseteq \conv(\Sol(\mt{D}^{1}))$.
	
	\smallskip
	For the second case, assume that $\xi^* \leq b$.	 		  
	Then, the if-condition in line 14 of Algorithm~\ref{alg:nonseparable DD} is satisfied, and node $\mt{u}_{n}$ is connected to the terminal node $\mt{t}$ of $\mt{D}^{2}$ via two arcs with label values $\mc{P}_n^2 \lb$ and $\mc{P}_n^2 \ub$. 
	Therefore, the solution set of $\mt{D}^{2}$ contains $2^n$ points encoded by all the $\mt{r}$-$\mt{t}$ paths of the DD, each composed of arcs with label values $\mc{P}_i^2 \lb$ or $\mc{P}_i^2 \ub$ for $i \in [n]$. 
	It is clear that these points correspond to the extreme points of the rectangular partition $\mc{P}^{2} = \prod_{i=1}^n \mc{P}_i^{2}$. 
	Pick one of these points, denoted by $\bar{\vc{x}}$. 
	We show that $\bar{\vc{x}} \in \conv(\Sol(\mt{D}^{1}))$. 	
	It follows from lines 1--12 of Algorithm~\ref{alg:nonseparable DD} that each layer $i \in [n]$ of $\mt{D}^1$ includes a single node $\mt{v}_i$. 
	Further, each node $\mt{v}_i$ is connected to $\mt{v}_{i+1}$ via two arcs with label values $\mc{P}_i^1 \lb$ and $\mc{P}_i^1 \ub$ for $i \in [n-1]$.
	To determine whether $\mt{v}_n$ is connected to the terminal node of $\mt{D}^{1}$, we need to calculate $\xi^*$ (which we refer to as $\dot{\xi}^*$ to distinguish it from that calculated for $\mt{D}^{2}$) according to line 9 of Algorithm~\ref{alg:nonseparable DD}.
	Since all the nodes $\mt{v}_1, \dotsc, \mt{v}_n$ are connected via the arcs described above, we conclude that the sub-domain of variable $x_i$ relative to node $\mt{v}_j$ for each $i \in [n-1]$ and $j > i$ is the entire variable domain $\mc{P}_i^{1}$.  
	Using an argument similar to that in the proof of Proposition~\ref{prop:nonseparable convex hull}, we write that $\dot{\xi}^* = \sum_{k=1}^q \dot{\eta}_k$, where $\dot{\eta}_k \leq g_k(\vc{x}_{H_k})$ for all $x_j \in \mc{P}_j^{1}$ with $j \in H_k$, which is obtained from the lower bound calculation rule employed for this algorithm.	
	We can similarly calculate the value of $\xi^*$ for $\mt{D}^{2}$ as $\xi^* = \sum_{k=1}^q \eta_k \leq b$, where the inequality holds by the assumption for this case, and where $\eta_k \leq g_k(\vc{x}_{H_k})$ for all $x_j \in \mc{P}_j^{2}$ with $j \in H_k$, which is obtained from the lower bound calculation rule employed for this algorithm.	
	On the other hand, because $\mc{P}^{2} \subseteq \mc{P}^{1}$, we have that $\mc{P}_i^1 \lb \leq \mc{P}_i^2 \lb \leq \mc{P}_i^2 \ub \leq \mc{P}_i^1 \ub$ for each $i \in [n]$.
	As a result, due to consistency property of the lower bound calculation rule, we have that $\dot{\eta}_k \leq \eta_k$ for each $k \in [q]$.
	Combining the above results, we obtain that $\dot{\xi}^* = \sum_{i=1}^q \dot{\eta}_q \leq \sum_{i=1}^q \eta_q \leq b$.
	Therefore, the if-condition in line 14 of Algorithm~\ref{alg:nonseparable DD} is satisfied for $\mt{D}^{1}$, and thus $\mt{v}_n$ is connected to the terminal node of $\mt{D}^{1}$ via two arcs with label values $\mc{P}_n^1 \lb$ and $\mc{P}_n^1 \ub$.
	Consequently, $\Sol(\mt{D}^{1})$ includes all extreme points of the rectangular partition $\mc{P}^{1}$ encoded by the $\mt{r}$-$\mt{t}$ paths of this DD.
	Since $\mc{P}^{2} \subseteq \mc{P}^{1}$, the extreme point $\bar{\vc{x}}$ of $\mc{P}^{2}$ is in $\conv(\Sol(\mt{D}^{1}))$, proving the result.
		\qed 
\end{proof}


Although Proposition~\ref{prop:spatial BB nonseparable} implies that the dual bounds obtained by our proposed DD-based outer approximation framework can improve through \SBB as a result of partitioning the variables' domain, an additional property of the employed lower bound calculation rules is needed to guarantee convergence to the global optimal value of the problem, as described next.

\begin{definition} \label{def:convergent lower bound}
	Consider a function $g(\vc{x}):\mc{D} \to \Re$, where $C \subseteq [n]$ and $I = [n] \setminus C$ represent the index sets of continuous and integer variable, respectively, and where $\mc{D} = \prod_{i=1}^n \mc{D}_i$ with $\mc{D}_i = [\mc{D}_i \lb, \mc{D}_i \ub]$ for $i \in C$ and $\mc{D}_i = [\mc{D}_i \lb, \mc{D}_i \ub] \cap \Z$ for $i \in I$.
	Consider a lower bound calculation rule that outputs a lower bound $\eta(\bar{\mc{D}})$ for $g(\vc{x})$ over a box domain $\bar{\mc{D}} \subseteq \mc{D}$.
	We say that this lower bound calculation rule is \textit{convergent with respect to} $g(\vc{x})$ \textit{over} $\mc{D}$ if (i) it is consistent with respect to $g(\vc{x})$ over $\mc{D}$, and (ii) $\lim_{j \to \infty} \eta(\mc{D}^j) = g(\bar{\vc{x}})$ for any nested sequence of box domains $\{\mc{D}^j\}_{j=1}^{\infty}$ with $\mc{D}^j \subseteq \mc{D}$ that converges (in the Hausdorff sense) to a singleton set $\{\bar{x}\}$, i.e., $\{\mc{D}^j\} \searrow \{\bar{\vc{x}}\}$.  
\end{definition}	

As the next step, Proposition~\ref{prop:convergence nonseparable} gives the convergence results for constraints with general non-separable terms.

\begin{proposition}	\label{prop:convergence nonseparable}
	Consider a function $g(\vc{x}):\mc{P} \to \Re$, where $g(\vc{x}) = \sum_{j=1}^q g_j(\vc{x}_{H_j})$, with each $g_j(\vc{x}_{H_j})$ being a non-separable function that contains variables with indices in $H_j \subseteq [n]$.
	Let $C \subseteq [n]$ and $I = [n] \setminus C$ represent the index sets of continuous and integer variable, respectively.
	Consider $\mc{P} = \prod_{i=1}^n \mc{P}_i$, where $\mc{P}_i = [\mc{P}_i \lb, \mc{P}_i \ub]$ for $i \in C$ and $\mc{P}_i = [\mc{P}_i \lb, \mc{P}_i \ub] \cap \Z$ for $i \in I$.
	Define $\mathcal{F}^j = \{\vc{x} \in \mc{P}^j \, | \, g(\vc{x}) \leq b \}$ for any $j \in \N$, where $\mc{P}^j = \prod_{i=1}^n \mc{P}^j_i \subseteq \mc{P}$.
	For $j \in \N$, let $\mt{D}^j$ be the DD representing $\mathcal{F}^j$, which is constructed via Algorithm~\ref{alg:nonseparable DD} or \ref{alg:relaxed nonseparable DD} for the single sub-domain partition $\mc{P}^j_i$ for $i \in [n]$ using a lower bound calculation rule convergent with respect to $g_k(\vc{x}_{H_k})$ over $\prod_{l \in H_k}\mc{P}_l$ for each $k \in [q]$.
	Assume that $\{\mc{P}^1, \mc{P}^2, \dotsc\}$, with $\mc{P}^j \subseteq \mc{P}$, is a nested sequence of rectangular partitions of the variables domain created through the \SBB process, i.e., $\mc{P}^j \supseteq \mc{P}^{j+1}$ for each $j \in \N$.
	Let $\tilde{\vc{x}} \in \Re^n$ with $\tilde{x}_i \in \Z$ for $i \in I$ be the point in a singleton set to which the above sequence converges (in the Hausdorff sense), i.e., $\{\mc{P}^j\} \searrow \{\tilde{\vc{x}}\}$.
	Then, the following statements hold:
	\begin{itemize}
		\item[(i)] If $g(\tilde{\vc{x}}) \leq b$, then $\big\{\conv(\Sol(\mt{D}^j))\big\} \searrow \{\tilde{\vc{x}}\}$.
		\item[(ii)] If $g(\tilde{\vc{x}}) > b$, then there exists $m \in \N$ such that $\Sol(\mt{D}^j) = \emptyset$ for all $j \geq m$.
	\end{itemize}	
\end{proposition}

\begin{proof}
	\begin{itemize}
		\item [(i)] Assume that $g(\tilde{\vc{x}}) \leq b$.	
		Since there is only one sub-domain partition for each variable, the DDs constructed via Algorithm~\ref{alg:nonseparable DD} and \ref{alg:relaxed nonseparable DD} are the same.
		Thus, we show the result assuming Algorithm~\ref{alg:nonseparable DD} is used.
		Consider $j \in \N$.
		Note that $\mathcal{F}^j \subseteq \conv(\mathcal{F}^j) \subseteq \conv(\Sol(\mt{D}^j))$ according to Proposition~\ref{prop:nonseparable convex hull}.
		We prove that $\Sol(\mt{D}^j) \subseteq \mc{P}^j$.
		There are two cases.
		For the first case, assume that the if-condition in line 14 of Algorithm~\ref{alg:nonseparable DD} is not satisfied.
		It implies that there are no $\mt{r}$-$\mt{t}$ paths in $\mt{D}^j$, i.e., $\Sol(\mt{D}^j) = \emptyset \subseteq \mc{P}^j$.
		For the second case, assume that the if-condition in line 14 of Algorithm~\ref{alg:nonseparable DD} is satisfied.
		Then, $\Sol(\mt{D}^j)$ contains the points encoded by all $\mt{r}$-$\mt{t}$ paths in $\mt{D}^j$ composed of arcs with label values $\mc{P}_i^j \lb$ or $\mc{P}_i^j \ub$ for each $i \in [n]$, i.e., $\Sol(\mt{D}^j) \subseteq \mc{P}^j$.
		As a result, $\conv(\Sol(\mt{D}^j)) \subseteq \mc{P}^j$.
		Because $\{\mc{P}^1, \mc{P}^2, \dotsc \}$ is a nested set sequence, it follows from Proposition~\ref{prop:spatial BB nonseparable} that the sequence $\{\conv(\Sol(\mt{D}_1)), \conv(\Sol(\mt{D}_2)), \dotsc \}$ is also nested, i.e., $\conv(\Sol(\mt{D}^j)) \supseteq \conv(\Sol(\mt{D}^{j+1}))$ for $j \in \N$.
		On the other hand, we can write $\mathcal{F}^j = \{\vc{x} \in \Re^n \, | \, g(\vc{x}) \leq b \} \cap \mc{P}^j$ by definition.
		Since $\{\mc{P}^j\} \searrow \{\tilde{\vc{x}}\}$, we obtain that $\{\mathcal{F}^j\} \searrow \{\vc{x} \in \Re^n \, | \, g(\vc{x}) \leq b \} \cap \{\tilde{\vc{x}}\} = \{\tilde{\vc{x}}\}$ since $g(\tilde{\vc{x}}) \leq b$ by assumption.
		Therefore, based on the previous arguments, we can write that $\mathcal{F}^j \subseteq \conv(\Sol(\mt{D}^j)) \subseteq \mc{P}^j$. 
		Because $\{\mathcal{F}^j\} \searrow \{\tilde{\vc{x}}\}$ and $\{\mc{P}^j\} \searrow \{\tilde{\vc{x}}\}$, we conclude that $\big\{\conv(\Sol(\mt{D}^j))\big\} \searrow \{\tilde{\vc{x}}\}$.
		
		\item[(ii)] Assume that $g(\tilde{\vc{x}}) > b$.	
		For each DD $\mt{D}^j$ for $j \in \N$, using a similar approach to that of Proposition~\ref{prop:spatial BB nonseparable}, we can calculate the value $\xi^* = \sum_{k=1}^q \eta_k$ at layer $n$ of the DD in line 9 of Algorithm~\ref{alg:nonseparable DD}, where $\eta_k \leq g_k(\vc{x}_{H_k})$ for all $x_i \in \mc{P}^j_i$ for each $i \in H_k$, which is obtained from the lower bound calculation rule employed for this algorithm.
		In this relation, we have used the fact that $\mt{D}^j$ has a unit width, thus the sub-domain of each variable $x_i$ relative to the single node at any layer of the DD is the entire domain $\mc{P}^j_i$. 	
		The assumption of this case implies that $g(\tilde{\vc{x}}) = \sum_{k=1}^q g_k(\tilde{\vc{x}}_{H_k}) > b$. 
		Define $\epsilon =  \frac{\sum_{k=1}^q g_k(\tilde{\vc{x}}_{H_k}) - b}{q} > 0$.
		For each $k \in [q]$, by definition of convergence for the lower bound calculation rule, the lower bounds $\eta_k$ of $g_k(\vc{x}_{H_k})$ computed in line 6 of Algorithm~\ref{alg:nonseparable DD} monotonically converge to $g_k(\tilde{\vc{x}}_{H_k})$ as the domain partitions $\mc{P}_i^j$ converge to $\{\tilde{x}_i\}$.
		Therefore, there exists $m_i \in \N$ such that $\eta_k > g_k(\tilde{\vc{x}}_{H_k}) - \epsilon$ computed over the domain partition $\mc{P}_i^{j}$ for all $j \geq m_i$.  
		Pick $m = \max_{i \in [n]} m_i$.
		The value of $\xi^*$ for $\mt{D}^m$ is calculated as $\xi^* = \sum_{i=1}^q \eta_k > \sum_{i =1}^q \big(g_k(\tilde{\vc{x}}_{H_k}) - \epsilon \big)= \sum_{i=1}^q g_k(\tilde{\vc{x}}_{H_k}) - q \epsilon = b$, where the inequality follows from the value of $\eta_k$ computed above, and the last equality is due to the definition of $\epsilon$ given previously.
		Since $\xi^* > b$, the if-condition in line 14 of Algorithm~\ref{alg:nonseparable DD} is not satisfied, and thus the single node $\mt{v}_n$ at layer $n$ of $\mt{D}^m$ is not connected to the terminal node of this DD, implying that $\Sol(\mt{D}^m) = \emptyset$.
		Finally, it follows from Proposition~\ref{prop:spatial BB nonseparable} that $\Sol(\mt{D}^j) \subseteq \conv(\Sol(\mt{D}^j)) \subseteq \conv(\Sol(\mt{D}^m)) = \emptyset$, for all $j > m$, proving the result. 
	\end{itemize} 
		\qed
\end{proof}

The result of Proposition~\ref{prop:convergence nonseparable} shows that the convex hull of the solution set, as represented by the DDs constructed through the proposed convexification technique, converges to the feasible region of the underlying MINLP constraint during the \SBB process.
This guarantees convergence to the global optimal value of the MINLP (if one exists), as implemented in Algorithm~\ref{alg:global}.

\smallskip
Since the convergence results above depend on the convergence properties of the lower bound calculation rules used in the DD construction method, we conclude this section by outlining the conditions required to achieve these properties. 
First, we demonstrate that a necessary condition for this property pertains to a variant of lower semicontinuity in the functions defined over the space of their continuous variables, if such variables are present.
Consider a function $g(\vc{x}):\Re^n \to \Re$, where $I \subset [n]$ and $C = [n] \setminus I$ represent the index sets of integer and continuous variables, respectively.
Following the Definition~\ref{def:monotone}, we denote by $g(\vc{x}_C, \bar{\vc{x}}_{I}):\Re^{|C|} \to \Re$ the restriction of $g(\vc{x})$ in the space of $\vc{x}_C$, where variables $x_k$ are fixed at value $\bar{x}_k$ for all $k \in I$.
Further, we say that $g(\vc{x})$ is \textit{lower semicontinuous over a domain} $\mc{D} \subseteq \Re^n$ if for any point $\bar{\vc{x}} \in \mc{D}$ and any $\epsilon > 0$, there exists $\delta > 0$ such that $g(\vc{x}) > g(\bar{\vc{x}}) - \epsilon$ for every $\vc{x} \in \mc{D}$ with $||\vc{x} - \bar{\vc{x}}||_2 < \delta$.

\begin{proposition} \label{prop:semicont}
	Consider a function $g(\vc{x}):\mc{D} \to \Re$, where $I \subset [n]$ and $C = [n] \setminus I$ represent the index sets of integer and continuous variables, respectively, and where $\mc{D} = \prod_{i=1}^n \mc{D}_i$ with $\mc{D}_i = [\mc{D}_i \lb, \mc{D}_i \ub]$ for $i \in C$ and $\mc{D}_i = [\mc{D}_i \lb, \mc{D}_i \ub] \cap \Z$ for $i \in I$.
	Consider a lower bound calculation rule that outputs a lower bound $\eta(\bar{\mc{D}})$ for $g(\vc{x})$ over a box domain $\bar{\mc{D}} \subseteq \mc{D}$.
	If this lower bound calculation rule is convergent with respect to $g(\vc{x})$ over $\mc{D}$, then $g(\vc{x}_C, \bar{\vc{x}}_{I})$ is lower semicontinuous over $\prod_{i \in C} \mc{D}_i$ for any $\bar{\vc{x}}_{I} \in \prod_{i \in I} \mc{D}_i$. 
\end{proposition}

\begin{proof}
	Assume by contradiction that there exists $\bar{\vc{x}}_{I} \in \prod_{i \in I} \mc{D}_i$ such that $g(\vc{x}_C, \bar{\vc{x}}_{I})$ is not lower semicontinuous over $\prod_{i \in C} \mc{D}_i$.
	Therefore, there exist $\tilde{\vc{x}}_C \in \prod_{i \in C} \mc{D}_i$ and $\epsilon > 0$ such that, for any $\delta > 0$, there is a point $\hat{\vc{x}}_C^{\delta} \in \prod_{i \in C} \mc{D}_i$ with $g(\hat{\vc{x}}_C^{\delta}, \bar{\vc{x}}_{I}) \leq g(\tilde{\vc{x}}_C, \bar{\vc{x}}_{I}) - \epsilon$ and $||\hat{\vc{x}}_C^{\delta} - \tilde{\vc{x}}_C||_2 < \delta$. 
	Consider a sequence $\{\delta^j\}$ with $\delta^j = 1/j$ for $j \in \N$.
	Define a sequence of box domains $\{\mc{D}^j\}$ with $\mc{D}^j = \prod_{i=1}^n \mc{D}^j_i$ where $\mc{D}^j_i = [\bar{x}_i, \bar{x}_i]$ for each $i \in I$ and $\mc{D}^j_i = [\tilde{x}_i - \delta^j, \tilde{x}_i + \delta^j] \cap [\mc{D}_i \lb, \mc{D}_i \ub]$ for each $i \in C$.
	It is clear that $\{\mc{D}^j\}$ converges to $\{(\tilde{\vc{x}}_C, \bar{\vc{x}}_{I})\}$.
	Furthermore, it follows from the definition of $\hat{\vc{x}}_C^{\delta}$ that $(\hat{\vc{x}}_C^{\delta^j}, \bar{\vc{x}}_{I}) \in \mc{D}^j$ for each $j \in \N$.
	As a result, the lower bound $\eta(\mc{D}^j)$ obtained by the lower bound calculation rule satisfies $\eta(\mc{D}^j) \leq g(\hat{\vc{x}}_C^{\delta^j}, \bar{\vc{x}}_{I}) \leq g(\tilde{\vc{x}}_C, \bar{\vc{x}}_{I}) - \epsilon < g(\tilde{\vc{x}}_C, \bar{\vc{x}}_{I}) - \frac{\epsilon}{2}$, where the second inequality follows from the contradiction assumption, and the last inequality holds because $\epsilon > 0$ by assumption.
	This is a contradiction to the assumption that the considered lower bound calculation rule is convergent with respect to $g(\vc{x})$ over $\mc{D}$ as for the domain sequence $\{\mc{D}^j\}$, we must have $\lim_{j \to \infty} \eta(\mc{D}^j) \neq g(\tilde{\vc{x}}_C, \bar{\vc{x}}_{I})$.
		\qed
\end{proof}

Next, we show that the lower bound calculation rules introduced in Section~~\ref{sub:lower bound} possess the convergence property when applied to functions that satisfy the necessary condition outlined in Proposition~\ref{prop:semicont}.
In other words, as long as this functional property for the MINLP is fulfilled, our proposed lower bound calculation rules guarantee convergence to a global solution.
Considering that the lower semicontinuity of Proposition~\ref{prop:semicont} holds for a broad range of functions commonly used in MINLP models, including test instances in the MINLP library, our proposed framework provides a powerful tool for globally solving various families of MINLPs.

\begin{proposition} \label{prop:monotone convergent}
	Consider a monotone function $g(\vc{x}):\mc{D} \to \Re$, where $I \subseteq [n]$ and $C = [n] \setminus I$ represent the index sets of integer and continuous variables, respectively, and where $\mc{D} = \prod_{i=1}^n \mc{D}_i$ with $\mc{D}_i = [\mc{D}_i \lb, \mc{D}_i \ub]$ for $i \in C$ and $\mc{D}_i = [\mc{D}_i \lb, \mc{D}_i \ub] \cap \Z$ for $i \in I$.
	Assume that $g(\vc{x}_C, \bar{\vc{x}}_{I})$ is lower semicontinuous over $\prod_{i \in C} \mc{D}_i$ for any $\bar{\vc{x}}_{I} \in \prod_{i \in I} \mc{D}_i$.
	Then, the lower bound calculation rule described in Proposition~\ref{prop:monotone} is convergent with respect to $g(\vc{x})$ over $\mc{D}$.
\end{proposition}

\begin{proof}
	Let $\eta(\bar{\mc{D}})$ be the lower bound of $g(\vc{x})$ over $\bar{\mc{D}} \subseteq \mc{D}$ calculated by the considered lower bound calculation rule.
	It follows from the definition of Proposition~\ref{prop:monotone} that $\eta(\bar{\mc{D}}) = \min_{\vc{x} \in \bar{\mc{D}}} \{g(\vc{x})\}$.
	This definition implies that $\eta(\mc{D}^1) \leq \eta(\mc{D}^2)$ for any $\mc{D}^2 \subseteq \mc{D}^1 \subseteq \mc{D}$, proving condition (i) of convergence property.
	To prove condition (ii), consider a nested domain sequence $\{\mc{D}^j\}$ with $\mc{D}^j = \prod_{i=1}^n \mc{D}^j_1 \subseteq \mc{D}$ for $j \in \N$ such that $\{\mc{D}^j\} \searrow \{\tilde{\vc{x}}\}$ for some $\tilde{\vc{x}} \in \mc{D}$.
	On the one hand, since $\tilde{\vc{x}} \in \mc{D}$, we must have $\tilde{x}_i \in \Z$ for $i \in I$.
	Further, for each $i \in I$, we have $\mc{D}^j_i \subseteq \Z$ for all $j \in \N$ by definition.
	Thus, there exist $\bar{m} \in \N$ such that $\mc{D}^j_i = [\tilde{x}_i, \tilde{x}_i]$ for each $i \in I$ and $j \geq \bar{m}$.
	On the other hand, it follows from the lower semicontinuity of $g(\vc{x}_C, \tilde{\vc{x}}_{I})$ over $\prod_{i \in C} \mc{D}_i$ that, for any $\epsilon > 0$, there exists $\hat{m} \in \N$ such that $g(\vc{x}_C, \tilde{\vc{x}}_{I}) > g(\tilde{\vc{x}}_C, \tilde{\vc{x}}_{I}) - \epsilon$ for each $\vc{x}_C \in \prod_{i \in C} \mc{D}^j_i$ for all $j \geq \hat{m}$.
	Define $m = \max\{\bar{m}, \hat{m}\}$.
	For all $\vc{x} \in \mc{D}^j$ with $j \geq m$, we can write $g(\vc{x}) = g(\vc{x}_C, \tilde{\vc{x}}_{I}) > g(\tilde{\vc{x}}_C, \tilde{\vc{x}}_{I}) - \epsilon = g(\tilde{\vc{x}}) - \epsilon$, where the first equality follows from the fact that $x_i = \tilde{x}_i$ for $i \in I$, the inequality is due to the relation obtained previously, and the last equality holds because $(\tilde{\vc{x}}_C, \tilde{\vc{x}}_{I}) = \tilde{\vc{x}}$ by definition.
	As a result, $\eta(\mc{D}^j) = \min_{\vc{x} \in \mc{D}^j} \{g(\vc{x})\} > g(\tilde{\vc{x}}) - \epsilon$ for all $j \geq m$.
	Since this result holds for any $\epsilon > 0$, we conclude that $\lim_{j \to \infty} \eta(\mc{D}^j) = g(\tilde{\vc{x}})$, proving the result.
		\qed	
\end{proof}

\begin{proposition} \label{prop:reindex convergent}
	Consider a function $g(\vc{x}):\mc{D} \to \Re$, where $I \subseteq [n]$ and $C = [n] \setminus I$ represent the index sets of integer and continuous variables, respectively, and where $\mc{D} = \prod_{i=1}^n \mc{D}_i$ with $\mc{D}_i = [\mc{D}_i \lb, \mc{D}_i \ub]$ for $i \in C$ and $\mc{D}_i = [\mc{D}_i \lb, \mc{D}_i \ub] \cap \Z$ for $i \in I$.
	Let $g^{\text{rx}}(\vc{y}):\Re^{p} \to \Re$ be the re-indexed function of $g(\vc{x})$ with re-index mapping $R(.)$.
	Assume that $g^{\text{rx}}(\vc{y})$ is monotone over the box domain described by $\mc{P} = \prod_{j=1}^p \mc{P}_j$ where $\mc{P}_j = \mc{D}_{R(j)}$ for $j \in [p]$.
	Define the index set of continuous variables in $g^{\text{rx}}(\vc{y})$ as $\dot{C} = \big\{ j \in [p] \big| R(j) \in C\big\}$, and define the index set of integer variables in $g^{\text{rx}}(\vc{y})$ as $\dot{I} = [p] \setminus \dot{C}$.
	Assume that $g^{\text{rx}}(\vc{y}_{\dot{C}}, \bar{\vc{y}}_{\dot{I}})$ is lower semicontinuous over $\prod_{j \in \dot{C}} \mc{P}_j$ for any $\bar{\vc{y}}_{\dot{I}} \in \prod_{j \in \dot{I}} \mc{P}_j$.	
	Then, the lower bound calculation rule described in Proposition~\ref{prop:reindexing} is convergent with respect to $g(\vc{x})$ over $\mc{D}$.
\end{proposition}

\begin{proof}
	Since $g^{\text{rx}}(\vc{y})$ is monotone, it follows from the definition of the lower bound calculation rule of Proposition~\ref{prop:reindexing} that $\eta(\mc{D}) = \min_{\vc{y} \in \mc{P}} \{g^{\text{rx}}(\vc{y})\}$.
	To prove condition (i) of the convergence property, consider box domains $\mc{D}^1$ and $\mc{D}^2$ such that $\mc{D}^1 \subseteq \mc{D}^2 \subseteq \mc{D}$.
	Define $\mc{P}^1 = \prod_{j=1}^p \mc{D}^1_{R(j)}$ and $\mc{P}^2 = \prod_{j=1}^p \mc{D}^2_{R(j)}$.
	It follows that $\mc{P}^1 \subseteq \mc{P}^2$. 
	Therefore, we can write that $\eta(\mc{D}^1) = \min_{\vc{y} \in \mc{P}^1} \{g^{\text{rx}}(\vc{y}) \} \geq \min_{\vc{y} \in \mc{P}^2} \{g^{\text{rx}}(\vc{y})\} = \eta(\mc{D}^2)$, which proves condition (i).	
	For condition (ii) of the convergence property, since $g^{\text{rx}}(\vc{y})$ is monotone and $g^{\text{rx}}(\vc{y}_{\dot{C}}, \bar{\vc{y}}_{\dot{I}})$ is lower semicontinuous over $\prod_{j \in \dot{C}} \mc{P}_j$ for any $\bar{\vc{y}}_{\dot{I}} \in \prod_{j \in \dot{I}} \mc{P}_j$, Proposition~\ref{prop:monotone convergent} implies that the lower bound calculation rule that outputs $\dot{\eta}(\mc{P}) = \min_{\vc{y} \in \mc{P}} \{g^{\text{rx}}(\vc{y})\}$ is convergent with respect to $g^{\text{rx}}(\vc{y})$ over $\mc{P}$.
	In other words, for any nested domain sequence $\{\mc{P}^k\}$, with $\mc{P}^k \subseteq \mc{P}$ for $k \in \N$, that converges to $\tilde{\vc{y}}$, we have $\lim_{k \to \infty} \dot{\eta}(\mc{P}^k) = g^{\text{rx}}(\tilde{\vc{y}})$.
	Consider a nested domain sequence $\{\mc{D}^k\}$, with $\mc{D}^k \subseteq \mc{D}$ for $k \in \N$, that converges to $\tilde{\vc{x}}$.
	Define $\dot{\mc{P}}^k = \prod_{j=1}^p \mc{D}^k_{R(j)}$ for each $k \in \N$.
	It is clear that $\{\dot{\mc{P}}^k\} \searrow \{\dot{y}\}$ where $\dot{y}_j = \tilde{x}_{R(j)}$ for each $j \in [p]$.
	Thus, we obtain $\lim_{k \to \infty} \dot{\eta}(\dot{\mc{P}}^k) = g^{\text{rx}}(\dot{\vc{y}})$ by the above definition.
	Using the fact that $\eta(\mc{D}^k) = \min_{\vc{y} \in \dot{\mc{P}}^k} \{g^{\text{rx}}(\vc{y})\} = \dot{\eta}(\dot{\mc{P}}^k)$ by definition of the considered lower bound calculation rule, we conclude that $\lim_{k \to \infty} \eta(\mc{D}^k) = g^{\text{rx}}(\dot{\vc{y}}) = g(\tilde{\vc{x}})$, where the last equality follows from the definition of re-indexed functions that preserve the function values at each given point.
	This shows that condition (ii) of the convergence property is satisfied.
		\qed	
\end{proof}

\section{Computational Results} \label{sec:computation}

In this section, we present numerical results based on benchmark instances from the MINLP Library \cite{minlplib} to demonstrate the effectiveness and capabilities of our global solution framework, compared to state-of-the-art global solvers. 
Since previous studies utilizing DD-based outer approximation \cite{davarnia:va:2020,davarnia:2021} have primarily focused on challenging problem classes that existing global solvers can handle, albeit with optimality gaps, this paper focuses on complementary problem classes that are unsolvable by current global solvers, marking them as the most challenging problems in the MINLP Library. 
As discussed in Section~\ref{sub:lower bound}, one of the main advantages of our DD-based global solution framework, compared to existing methods, is its ability to model and solve a broader class of MINLPs, including those with complex structures that are not amenable to conventional convexification methods, such as the factorable decomposition technique that is widely used in existing solvers. 
To demonstrate this capability, in this section, we present computational experiments on benchmark test instances from the MINLP Library that contain functional forms not admissible by global solvers, such as BARON, the leading commercial solver, and SCIP, the leading open-source solver.

\subsection{Algorithmic Settings} \label{sub:setting}

The numerical results presented in this section are obtained on a Windows $11$ ($64$-bit) operating system, $64$ GB RAM, $3.8$ GHz AMD Ryzen CPU. 
The DD-ECP Algorithm is written in Julia v1.9 via JuMP v1.11.1, and the outer approximation models are solved with CPLEX v22.1.0.
In this section, we present the general settings for the algorithms used in our solution framework.

\smallskip
We use Algorithm~\ref{alg:global} to solve the MINLP instances reformulated into the problem form described in \eqref{eq:MINLP-1}--\eqref{eq:MINLP-3}.
\black{To conform to this structure, any equality constraints are handled by splitting them into two inequalities.}
Since the model studied in this paper is bounded, for instances with variables that lack explicit bounds, we infer valid bounds based on the constraints of the model.
To calculate the optimality gap, we use the primal bound reported for each instance in the MINLP Library. The stopping criteria employed in $\text{Stop\_Flag}$ in Algorithm~\ref{alg:global} are a remaining optimality gap of $0.05$ or an elapsed time of $5000$ seconds, whichever occurs first.
The initial LP relaxation $LP$ for each instance is obtained by removing all nonlinear constraints.
The pruning rules in $\text{Prune\_Node}$ include: (i) the dual bound obtained at a node is smaller than the best current primal bound; (ii) the outer approximation is infeasible; (iii) the DD constructed for any constraints is infeasible; and (iv) the optimal solution of the outer approximation satisfies all constraints.

\smallskip
For the $\mt{Construct\_DD}$ oracle, we use Algorithm~\ref{alg:relaxed nonseparable DD} for constructing DDs for general non-separable constraints.
In these algorithms, we create the sub-domain partitions for each variable $x_i$ for $i \in [n]$ such that the entire variable domain is divided into 50 intervals of equal length, i.e., $|L_i| = 50$.
We impose a default width limit of $\omega = 5000$.
To merge nodes at each layer, we apply the the merging policy $\mt{Merge}^g(.)$ as described in Section~\ref{sub:merging}.
The state values at the DD nodes are computed using the lower bound calculation rules based on the monotonicity property and re-indexing techniques outlined in Section~\ref{sub:lower bound}.

\smallskip
For the $\mt{Outer\_Approx}$ oracle, we use the subgradient-type method of Algorithm~\ref{alg:subgradient} to generate cutting planes that are added to the outer approximation model.
For this algorithm, we set a constant step size rule $\rho = 1$ and use the origin as the starting point for the subgradient algorithm.
The termination criterion is defined by the number of iterations, which is set to $50$.

\smallskip
For the $\mt{Branch}$ oracle, after obtaining the optimal solution of the outer approximation model, we select the variable whose optimal solution lies closest to the center of its domain interval. 
To continue the \BB process, we apply a node selection rule that prioritizes the node with the largest dual bound as the next candidate.

\subsection{Test Instances} \label{sub:test}

In this section, we present computational results for various benchmark instances from the MINLP Library. 
These instances feature complex functional structures that cannot be handled by existing global solvers like BARON and SCIP, and are therefore considered inadmissible/intractable. 
In contrast to global solvers, which fail to return dual bounds for these test instances, our DD-based global framework is capable of solving these problems and obtaining dual bounds, as shown in the tables for each instance.
To provide better insight into the structure of each model, a summary of the problem specifications, their area of application, and their sources of difficulty is presented in the following sub-sections.


\subsubsection{Test Instance: $\mt{quantum}$} \label{subsub:test1}

This problem has applications in quantum mechanics \cite{ogura1999post}. 
The test instance has $2$ continuous variables and $1$ nonlinear constraint. 
This nonlinear constraint includes polynomial, fractional, exponential, and gamma functions. 
The following constraint illustrates a complex structure used in this model that is inadmissible in the current solvers used.


\begin{equation*}
	\frac{-0.5 \, \sqrt{x_3} \, x_2^{\frac{1}{x_3}} \Gamma(2 - \frac{0.5}{x_3}) + 0.5 x_2^{\frac{-1}{x_3}} \Gamma( \frac{1.5}{x_3}) + x_2^{\frac{-2}{x_3}} \Gamma( \frac{2.5}{x_3})}{\Gamma(\frac{0.5}{x_3})} + z = 0,
\end{equation*}

\noindent where $\Gamma(.)$ is the gamma function. The performance of our DD-based solution framework is summarized in Table~\ref{tab:test1}.
The first two columns show the number of variables and constraints in each problem, respectively. 
The column labeled `Primal' presents the primal bound for the test instance, as reported in the MINLP Library. 
The dual bound obtained from our proposed global method is listed in the `Dual' column. 
The optimality gap is provided in the `Gap' column and is calculated as $\frac{\text{dual bound} - \text{primal bound}}{\text{primal bound}}$.
The next two columns, `Node Explored' and `Node Remaining,' represent the number of nodes explored and the number of nodes still open at the termination of the algorithm in the \BB tree. 
Finally, the last column shows the total solution time for the algorithm.

\begin{table}[h] 	
	\caption{Performance of the DD framework for test instance $\mt{quantum}$}			
	\label{tab:test1}
	\begin{center}
			\begin{tabular}{|r|r|r|r|r|r|r|r|}
				\hline
				\multicolumn{2}{|c|}{Problem Specs} & \multicolumn{3}{c|}{Gap Closure} & \multicolumn{2}{c|}{\BB Tree} & Time (s) \\ \cline{1-7}
				Var. \# & Con. \# \ & Primal & Dual & Gap & Node Explored & Node Remained &  \\
				\hline
				\hline 
				$2$  & $1$ & $-0.804$ & $-0.765$ & $0.05$ & $4$  & $1$  & $6.62$ \\				
				\hline
			\end{tabular}
	\end{center}
\end{table}

\vspace{-0.2in}
\subsubsection{Test Instance: $\mt{ann\_fermentation\_tanh}$} \label{subsub:test2}

This problem has applications in neural networks used to learn the fermentation process of gluconic acid, where the activation functions are represented by hyperbolic tangent operators \cite{schweidtmann2019deterministic}. 
The test instance has $12$ continuous variables and $10$ constraints. 
The nonlinear constraints include fractional and hyperbolic (trigonometric) functions. 
The following constraint illustrates a complex structure used in this model that is inadmissible in the current solvers used.

\begin{equation*}
	\tanh(x_{12}) - x_8 = 0,
\end{equation*}

\noindent where $\tanh(.)$ is the hyperbolic tangent function.
The performance of our proposed method when applied to this test instance is presented in Table~\ref{tab:test2}, with columns are defined similarly to those in Table~\ref{tab:test1}.

\begin{table}[h] 	
	\caption{Performance of the DD framework for test instance $\mt{ann\_fermentation\_tanh}$}			
	\label{tab:test2}
	\begin{center}
			\begin{tabular}{|r|r|r|r|r|r|r|r|}
				\hline
				\multicolumn{2}{|c|}{Problem Specs} & \multicolumn{3}{c|}{Gap Closure} & \multicolumn{2}{c|}{\BB Tree} & Time (s) \\ \cline{1-7}
				Var. \# & Con. \# \ & Primal & Dual & Gap & Node Explored & Node Remained &  \\
				\hline
				\hline 
				$12$  & $10$ & $99.93$ & $104.92$ & $0.05$ & $5104$  & $767$  & $25.74$ \\				
				\hline
			\end{tabular}
	\end{center}
\end{table}

\vspace{-0.2in}
\subsubsection{Test Instance: $\mt{fct}$} \label{subsub:test3}

This problem is included in the GAMS Model Library \cite{pinter1999lgo}. 
The test instance has $12$ continuous variables and $10$ constraints. 
The nonlinear constraints include absolute value, trigonometric, polynomial, and modulo functions. 
The following constraint shows a complex structure among the constraints used in this model, which is inadmissible in the current solvers.

\begin{equation*}
	\Big|\sin\big(4\, \text{mod}(x_2, \pi)\big)\Big| - x_3 = 0,
\end{equation*}

\noindent where $|.|$ is the absolute value function, and $\text{mod}(a, b)$ is the modulo operator with dividend $a$ and divisor $b$.
The performance of our proposed method when applied to this test instance is presented in Table~\ref{tab:test3}, with columns are defined similarly to those in Table~\ref{tab:test1}.
For this test instance, we did not calculate the remaining gap as the primal bound is zero.
Instead, we allowed the algorithm to run until it achieved a global optimal solution with a precision of $10^{-5}$ for the optimal value.

\begin{table}[h] 	
	\caption{Performance of the DD framework for test instance $\mt{fct}$}			
	\label{tab:test3}
	\begin{center}
			\begin{tabular}{|r|r|r|r|r|r|r|r|}
				\hline
				\multicolumn{2}{|c|}{Problem Specs} & \multicolumn{3}{c|}{Gap Closure} & \multicolumn{2}{c|}{\BB Tree} & Time (s) \\ \cline{1-7}
				Var. \# & Con. \# \ & Primal & Dual & Gap & Node Explored & Node Remained &  \\
				\hline
				\hline 
				$12$  & $10$ & $0.00$ & $6.83 \times 10^{-6}$  & -- & $591$  & $0$  & $1668.09$ \\				
				\hline
			\end{tabular}
	\end{center}
\end{table}

\vspace{-0.2in}
\subsubsection{Test Instance: $\mt{worst}$} \label{subsub:test4}

This problem has applications in statistical models used for portfolio optimization and risk management \cite{dahl1989some}. 
The test instance has $35$ continuous variables and $30$ constraints. 
The nonlinear constraints include polynomial, exponential, logarithm, fractional, and modulo functions. 
The following constraint shows a complex structure among the constraints used in this model, which is inadmissible in the current solvers.

\begin{equation*}
	e^{-0.33889 \, x_{32}} \times \big(\text{erf}(x_3) \, x_{21} - 95 \, \text{erf}(x_{10}) \big) - x_{23} = 0,
\end{equation*}

\noindent where $\text{erf}(.)$ is the error function calculated as the integral of the standard normal distribution.
The performance of our proposed method when applied to this test instance is presented in Table~\ref{tab:test4}, with columns are defined similarly to those in Table~\ref{tab:test1}.

\begin{table}[h] 	
	\caption{Performance of the DD framework for test instance $\mt{worst}$}			
	\label{tab:test4}
	\begin{center}
			\begin{tabular}{|r|r|r|r|r|r|r|r|}
				\hline
				\multicolumn{2}{|c|}{Problem Specs} & \multicolumn{3}{c|}{Gap Closure} & \multicolumn{2}{c|}{\BB Tree} & Time (s) \\ \cline{1-7}
				Var. \# & Con. \# \ & Primal & Dual & Gap & Node Explored & Node Remained &  \\
				\hline
				\hline 
				$35$  & $30$ & $-20762609$ & $-19583378$  & $0.05$ & $22$  & $3$  & $211.51$ \\				
				\hline
			\end{tabular}
	\end{center}
\end{table}

\vspace{-0.2in}
\subsubsection{Test Instance: $\mt{ann\_compressor\_tanh}$} \label{subsub:test5}

This problem has applications in learning compressor powers via neural networks \cite{schweidtmann2019deterministic}. 
This test instance has $97$ continuous variables and $96$ constraints. 
The nonlinear constraints in this problem include quadratic and hyperbolic (trigonometric) functions. 
The following constraint shows a complex structure among the constraints used in this model, which is inadmissible in the current solvers.

\begin{equation*}
	\tanh(x_{32}) - x_{10} = 0
\end{equation*}

The performance of our proposed method when applied to this test instance is presented in Table~\ref{tab:test5}, with columns are defined similarly to those in Table~\ref{tab:test1}.

\begin{table}[h] 	
	\caption{Performance of the DD framework for test instance $\mt{ann\_compressor\_tanh}$}			
	\label{tab:test5}
	\begin{center}
			\begin{tabular}{|r|r|r|r|r|r|r|r|}
				\hline
				\multicolumn{2}{|c|}{Problem Specs} & \multicolumn{3}{c|}{Gap Closure} & \multicolumn{2}{c|}{\BB Tree} & Time (s) \\ \cline{1-7}
				Var. \# & Con. \# \ & Primal & Dual & Gap & Node Explored & Node Remained &  \\
				\hline
				\hline 
				$97$  & $96$ & $-213100.0$ & $-22331.90$  & $0.05$ & $867$  & $52$  & $276.33$ \\				
				\hline
			\end{tabular}
	\end{center}
\end{table}

\vspace{-0.2in}
\subsubsection{Test Instance: $\mt{ann\_peaks\_tanh}$} \label{subsub:test6}

This problem has applications in neural networks \cite{schweidtmann2019deterministic}. 
This test instance has $100$ continuous variables and $99$ constraints. 
The nonlinear constraints in this problem include hyperbolic (trigonometric) functions. 
The following constraint shows a complex structure among the constraints used in this model, which is inadmissible in the current solvers.

\begin{equation*}
	\tanh(x_{55}) - x_6 = 0
\end{equation*}

The performance of our proposed method when applied to this test instance is presented in Table~\ref{tab:test6}, with columns are defined similarly to those in Table~\ref{tab:test1}.

\begin{table}[h] 	
	\caption{Performance of the DD framework for test instance $\mt{ann\_peaks\_tanh}$}			
	\label{tab:test6}
	\begin{center}
			\begin{tabular}{|r|r|r|r|r|r|r|r|}
				\hline
				\multicolumn{2}{|c|}{Problem Specs} & \multicolumn{3}{c|}{Gap Closure} & \multicolumn{2}{c|}{\BB Tree} & Time (s) \\ \cline{1-7}
				Var. \# & Con. \# \ & Primal & Dual & Gap & Node Explored & Node Remained &  \\
				\hline
				\hline 
				$100$  & $99$ & $6.56$ & $6.97$  & $0.05$ & $10$  & $7$  & $135.47$ \\				
				\hline
			\end{tabular}
	\end{center}
\end{table}

\vspace{-0.2in}

\subsubsection{Test Instance: $\mt{cesam2cent}$} \label{subsub:test7}

This problem has applications in information theory, econometrics, and estimating social accounting matrices using cross entropy methods \cite{golan1996maximum,judge2011information,robinson2001updating}. 
This test instance has $316$ continuous variables and $166$ constraints. 
The nonlinear constraints in this problem include polynomial, exponential, and cross entropy functions. 
The following constraint shows a complex structure among the constraints used in this model, which is inadmissible in the current solvers.

\begin{equation*}
	\sum_{i=160}^{316} \text{Centropy}(x_i, a_i) - z = 0,
\end{equation*}

\noindent where $a_i \in \Re$ is a constant, and $\text{Centropy}(.)$ is the cross-entropy function.
The performance of our proposed method when applied to this test instance is presented in Table~\ref{tab:test7}, with columns are defined similarly to those in Table~\ref{tab:test1}.

\begin{table}[h] 	
	\caption{Performance of the DD framework for test instance $\mt{cesam2cent}$}			
	\label{tab:test7}
	\begin{center}
			\begin{tabular}{|r|r|r|r|r|r|r|r|}
				\hline
				\multicolumn{2}{|c|}{Problem Specs} & \multicolumn{3}{c|}{Gap Closure} & \multicolumn{2}{c|}{\BB Tree} & Time (s) \\ \cline{1-7}
				Var. \# & Con. \# \ & Primal & Dual & Gap & Node Explored & Node Remained &  \\
				\hline
				\hline 
				$316$  & $166$ & $-0.507$ & $-0.481$  & $0.05$ & $8$  & $3$  & $4604.13$ \\				
				\hline
			\end{tabular}
	\end{center}
\end{table}

\vspace{-0.1in}
\section{Conclusion} \label{sec:conclusion}

We develop a novel graphical framework to globally solve general MINLPs. 
This paper details the key components of the framework, including (i) a method for constructing DDs that represent relaxations of the MINLP sets, (ii) a cut-generation technique that produces linear outer approximations of the underlying set, and (iii) a spatial branch-and-bound strategy that iteratively refines these approximations until convergence to a global optimal solution. 
Applicable to optimization problems of general structure, this framework represents the most comprehensive extension of previously developed DD-based approaches for MINLPs, addressing the longstanding need for a general-purpose DD-based method for globally solving MINLPs. 
Computational experiments on benchmark MINLP instances with complex structures, which are inadmissible in state-of-the-art global solvers, show the capabilities and effectiveness of the proposed framework.




\bibliographystyle{spbasic}
\bibliography{Mine,DD,MINLP,MINLP2}

\end{document}